\documentclass[reqno]{amsart}
\usepackage{amsmath}
\usepackage{amsxtra}
\usepackage{amscd}
\usepackage{amsthm}
\usepackage{amsfonts}
\usepackage{amssymb}
\usepackage{enumitem}
\usepackage{microtype}
\usepackage{graphicx}
\usepackage{epstopdf}
\usepackage{mathrsfs}
\usepackage{eucal} 
\usepackage{comment}
\usepackage[hidelinks, pagebackref, bookmarks=false]{hyperref}
\hypersetup{
    colorlinks,
    citecolor=black,
    filecolor=black,
    linkcolor=blue,
    urlcolor=black,
}
\usepackage{cite}
\usepackage{microtype}
\usepackage{url}

\input xy
\xyoption{all}

\numberwithin{equation}{subsection} 
\numberwithin{figure}{subsection} 

\makeatletter
\let\c@equation\c@figure
\makeatother

\newtheorem{theorem}[equation]{Theorem}
\newtheorem{corollary}[equation]{Corollary}
\newtheorem{lemma}[equation]{Lemma}
\newtheorem{prop}[equation]{Proposition}

\theoremstyle{remark}
\newtheorem{remark}[equation]{Remark}
\newtheorem{example}[equation]{Example}

\theoremstyle{definition}
\newtheorem{defn}[equation]{Definition}

\newcommand\nc{\newcommand}
\nc\on{\operatorname}

\newcommand\fp{{\mathfrak p}}
\newcommand\fq{{\mathfrak q}}
\newcommand\fm{{\mathfrak m}}

\nc\Hom{\on{Hom}}
\nc\Sections{\on{Sections}}
\nc\Sym{\on{Sym}}
\nc\Spec{\on{Spec}}
\nc\Specm{\on{Specm}}
\nc\ul{\underline}
\nc{\dfp}{\overset{\cdot}{\fp}}
\nc{\dfq}{\overset{\cdot}{\fq}}
\nc{\dfm}{\overset{\cdot}{\fm}}

\begin{document}

\title{Arithmetic representations of fundamental groups II:   finiteness}
\author{Daniel Litt}

\begin{abstract}
Let $X$ be a smooth curve over a finitely generated field $k$, and let $\ell$ be a prime different from the characteristic of $k$.  We analyze the dynamics of the Galois action on the deformation rings of mod $\ell$ representations of the geometric fundamental group of $X$.  Using this analysis, we prove analogues of the Shafarevich and Fontaine-Mazur finiteness conjectures for function fields over algebraically closed fields in arbitrary characteristic, and a weak variant of the Frey-Mazur conjecture for function fields in characteristic zero.  

For example, we show that if $X$ is a normal, connected variety over $\mathbb{C}$, the (typically infinite) set of representations of $\pi_1(X^{\text{an}})$ into $GL_n(\overline{\mathbb{Q}_\ell})$, which come from geometry, has no limit points.  As a corollary, we deduce that if $L$ is a finite extension of $\mathbb{Q}_\ell$, then the set of representations of $\pi_1(X^{\text{an}})$ into $GL_n(L)$, which arise from geometry, is finite.
\end{abstract}
\maketitle

\setcounter{tocdepth}{1}
\tableofcontents

\section{Introduction}
The purpose of this paper is to study the representations of the \'etale fundamental group of a variety $X$ over a finitely generated field $k$, via an analysis of the Galois action on $\pi_1^{\text{\'et}}(X_{\bar k}, \bar x)$.  This work was begun in \cite{litt}, which studied integral aspects of such representations.  In this paper we focus on finiteness results, motivated by the Shafarevich conjecture \cite[Satz 6]{faltings}, the Fontaine-Mazur finiteness conjectures \cite[Conjectures 2a and 2b]{fontaine-mazur}, and the Frey-Mazur conjecture (see the question at the end of the introduction of \cite{mazur3} for the original question, and the introduction of \cite{bakker-tsimerman} for a corrected statement).

Part of the goal of this work is to give \emph{anabelian} approaches to function field analogues of standard conjectures about representations of Galois groups of number fields, in the hope that these techniques can be transported to the number field setting. In particular, most of our main results have purely group-theoretic statements.

\subsection{Main results}
Let $k$ be a finitely generated field, and $X/k$ a curve (a smooth, separated, geometrically connected $k$-scheme of dimension $1$).  Choose an algebraic closure $\bar k$ of $k$, and let $\bar x$ be a geometric point of $X$. 
\begin{defn}
Let $\ell$ be a prime different from the characteristic of $k$, and let $L$ be an $\ell$-adic field (an algebraic extension of $\mathbb{Q}_\ell$ or the completion thereof). We say that a continuous representation $$\rho: \pi_1^{\text{\'et}}(X_{\bar k}, \bar x)\to GL_n(L)$$ is \emph{arithmetic} if there exists a finite extension $k'$ of $k$ and a continuous representation $$\tilde \rho: \pi_1^{\text{\'et}}(X_{k'}, \bar x)\to GL_n(L)$$ such that $\rho$ is a subquotient of $\tilde \rho|_{\pi_1^{\text{\'et}}(X_{\bar k}, \bar x)}.$
\end{defn}

The main examples of arithmetic representations are those \emph{arising from geometry} (see Definition \ref{arises-from-geometry-defn} for a precise definition):
\begin{example}
Suppose $$f: Y\to X_{\bar k}$$ is a smooth proper morphism of varieties over $\bar k$.  Then for any $i\geq 0$,  any subquotient of the monodromy representation $$\pi_1(X_{\bar k}, \bar x)\to GL((R^if_*\underline{\mathbb{Q}_\ell})_{\bar x}\otimes L)$$ is arithmetic.  (See Proposition \ref{comes-from-geometry} for a proof of a mild generalization of this fact.)
\end{example}
The purpose of this paper is to study arithmetic representations, and to apply this study to the understanding the representations which arise from geometry.
\subsubsection{Results on finiteness}
Our first main result is:
\begin{theorem}\label{main-theorem}
Let $c\in \mathbb{R}$ be such that $0<c\leq 1$.  Let $$\bar\rho: \pi_1^{\text{\'et}}(X_{\bar k}, \bar x)\to GL_n(\mathbb{F}_{\ell^r})$$ be a representation.  Then the set of semisimple arithmetic representations $$\rho: \pi_1^{\text{\'et}}(X_{\bar k}, \bar x) \to GL_n(\mathbb{C}_\ell)$$ with $$\on{Tr}(\rho)\equiv \on{Tr}(\bar \rho)\bmod \ell^c$$ is finite.
\end{theorem}
Here $\mathbb{C}_\ell:=\widehat{\overline{\mathbb{Q}_\ell}}$ is the completion of the algebraic closure of $\mathbb{Q}_\ell$.
\begin{remark}
A slight weakening of Theorem \ref{main-theorem} admits a more pithy statement. Namely, Theorem \ref{main-theorem} implies that the set of semisimple arithmetic representations of $\pi_1^{\text{\'et}}(X_{\bar k}, \bar x)\to GL_n(\mathbb{C}_\ell)$ has no limit points.  That is, if $\{\rho_i\}$ is a sequence of semisimple arithmetic representations with $\on{Tr}(\rho_i)\to \on{Tr}(\rho)$ pointwise, then the sequence $\{\rho_i\}$ is eventually constant.
\end{remark}

As a corollary, we have the following finiteness result:
\begin{corollary}\label{main-corollary}
Suppose $L$ is finite over $\mathbb{Q}_\ell$, with residue field $\mathbb{F}_{\ell^r}$, ring of integers $\mathscr{O}_L$, and maximal ideal $\mathfrak{m}_L$.  Then 
\begin{enumerate}
\item If $\bar\rho: \pi_1^{\text{\'et}}(X_{\bar k}, \bar x)\to GL_n(\mathbb{F}_{\ell^r})$ is a continuous representation, the set of semisimple arithmetic representations (resp.~representations which arise from geometry) $$\tilde\rho:  \pi_1^{\text{\'et}}(X_{\bar k}, \bar x)\to GL_n(L)$$ with $\on{Tr}(\tilde\rho)\equiv \on{Tr}(\rho) \bmod \mathfrak{m}_L$ is finite.
\item If $\on{char}(k)=0,$ the set of semisimple arithmetic representations (resp.~representations which arise from geometry) $$\rho: \pi_1^{\text{\'et}}(X_{\bar k}, \bar x) \to GL_n(L)$$ is finite.
\item If $\on{char}(k)>0,$ the set of semisimple tame arithmetic representations (resp.~tame representations which come from geometry) $$\rho: \pi_1^{\text{tame}}(X_{\bar k}, \bar x) \to GL_n(L)$$ is finite.
\end{enumerate}
\end{corollary}
This is the geometric analogue of the Shafarevich and Fontaine-Mazur finiteness conjectures referred to in the abstract.  
\begin{remark}
Theorem \ref{main-theorem} and the parts of Corollary \ref{main-corollary} about semisimple arithmetic representations are purely group-theoretic statements about the structure of the arithmetic fundamental group $\pi_1^{\text{\'et}}(X, \bar x)$.
\end{remark}
\begin{remark}
Note that if $L$ is replaced by $\overline{\mathbb{Q}_\ell}$ in  Corollary \ref{main-corollary} above, the statement is false, as may be seen by taking $X=\mathbb{G}_m$; in this case $\pi_1^{\text{\'et}}(X_{\bar k}, \bar x)\simeq \widehat{\mathbb{Z}}$, and the arithmetic representations (resp.~representations which arise from geometry) $\pi_1^{\text{\'et}}(X_{\bar k}, \bar x) \to \overline{\mathbb{Q}_\ell}^{\times}$ are precisely the characters of finite order, of which there are infinitely many.
\end{remark}
\begin{remark}
Corollary \ref{main-corollary}(3) is false without the tameness assumption.  For example, one may take $X=\mathbb{A}^1_k$; then $\on{Hom}(\pi_1^{\text{\'et}}(X_{\bar k}, \bar x), \mathbb{F}_p)$ is not finitely generated, and hence there are infinitely many representations $\pi_1^{\text{\'et}}(X_{\bar k}, \bar x)\to GL_p(\mathbb{Q}_\ell)$ with finite image.  Representations with finite image are always arithmetic (and always arise from geometry). One may, however, replace the tame fundamental group with any topologically finitely-generated quotient of $\pi_1^{\text{\'et}}(X_{\bar k}, \bar x)$.
\end{remark}
\begin{remark}
One may prove higher-dimensional analogues of Theorem \ref{main-theorem} and Corollary \ref{main-corollary} by reduction to the case of curves, via a Lefschetz argument.

Moreover, one may deduce results for varieties over arbitrary fields by a standard spreading-out and specialization argument. For example, Corollary \ref{main-corollary} implies that if $X$ is a connected, normal variety over $\mathbb{C}$, the set of reprsentations of $\pi_1(X^{\text{an}})$ into $GL_n(\mathbb{Q}_\ell)$, which arise from geometry, is finite.
\end{remark}
\subsubsection{A weak analogue of the Frey-Mazur conjecture}
Suppose now that $$\bar\rho: \pi_1^{\text{\'et}}(X_{\bar k}, \bar x)\to GL_n(\mathbb{F}_{\ell^r})$$ is geometrically irreducible.  Then Theorem \ref{main-theorem} implies (by \cite[Th\'{e}or\`{e}me 1]{carayol}) that given a semisimple arithmetic representation $$\tilde\rho: \pi_1^{\text{\'et}}(X_{\bar k}, \bar x)\to GL_n(\overline{\mathbb{Z}_\ell})$$ lifting $\bar\rho$, there exists $d\in \mathbb{Q}_{>0}$ such that for any semisimple arithmetic $\tilde\rho'$ with $$\tilde\rho'\simeq \tilde\rho \bmod \ell^d,$$ we have that $\tilde\rho'\simeq \tilde\rho$.  (Here $\overline{\mathbb{Z}_\ell}$ is the valuation ring of $\overline{\mathbb{Q}_\ell}$.)  That is, there is a ball of radius $\ell^{-d}$ around $\tilde\rho$ such that $\tilde\rho$ is the unique semisimple arithmetic lift of $\bar\rho$ within this ball. However, the proof of Theorem \ref{main-theorem} gives no way to effectively compute such a constant $d$.  

Our final main result gives a method to effectively compute such a constant $d>0$, in terms of cohomological invariants of $\tilde\rho$, as long as $\on{char}(k)=0$ and $\tilde\rho$ arises from geometry.  This is a weak  version of the Frey-Mazur conjecture for function fields (see e.g.~\cite{mazur3, bakker-tsimerman}), which asserts that a monodromy representation should be determined by its mod $\ell^d$ reduction, for $\ell^d$ large in terms of the geometric invariants of $X$ and the dimension of $\tilde\rho$, if $\tilde\rho$ arises from the Tate module of an Abelian $X$-scheme.

\begin{theorem}\label{ball-bound}
Let $X$ be a smooth, geometrically connected curve over a finitely generated field $k$ of characteristic zero, and let $\bar x$ be a geometric point of $X$.  Let $$\rho: \pi_1^{\text{\'et}}(X_{\bar k}, \bar x)\to GL_n(\overline{\mathbb{Q}_\ell})$$ be a representation which arises from geometry, lifting a geometrically irreducible residual representation $\bar\rho$.  Then there exists an explicit constant $N=N(c(\rho),\ell)$ such that any semisimple arithmetic representation $$\tilde\rho:  \pi_1^{\text{\'et}}(X_{\bar k}, \bar x)\to GL_n(\overline{\mathbb{Q}_\ell})$$ with $$\on{Tr}(\tilde\rho)\equiv \on{Tr}(\rho)\bmod \ell^N$$ satisfies $\rho\simeq \tilde\rho$.  Here $c(\rho)$ is defined as in Definition \ref{index-of-homothety}.
\end{theorem}
\begin{remark}
The statement of Theorem \ref{ball-bound} is not purely group-theoretic, as $\rho$ is required to arise from geometry. However, the proof only requires $\rho$ to be \emph{geometric in the sense of Fontaine-Mazur} --- see Remark \ref{fontaine-mazur-remark} for details.
\end{remark}
In general, the constant $c(\rho)$ appearing in Theorem \ref{ball-bound} may be bounded independently of $\ell$, assuming the Tate conjecture; in particular, assuming the Tate conjecture, if $\{\rho_\ell\}$ is a compatible system of representations of $\pi_1^{\text{\'et}}(X_{\bar k}, \bar x)$ arising from geometry, we have $N(c(\rho_\ell), \ell)\to 0$ as $\ell\to \infty$.  See Remark \ref{tate-remark} for details.

Without the Tate conjecture, we do not know how to estimate the constant $c(\rho)$ appearing in Theorem \ref{ball-bound} above; however, if $\rho$ has finite image, we may bound it using a result of Serre \cite[Lettre \`{a} Ken Ribet, p.~60]{serre-oeuvres}.  In this case, we have the following more uniform result, which is a strengthening of the main theorem of \cite{litt}:
\begin{theorem}\label{finite-image-thm}
Let $X, k, \bar x$ be as in Theorem \ref{ball-bound}. Let $$\rho: \pi_1^{\text{\'et}}(X_{\bar k}, \bar x)\to GL_n(\overline{\mathbb{Q}})$$ be an irreducible representation which factors through a finite quotient $G$ of $\pi_1^{\text{\'et}}(X_{\bar k}, \bar x)$.  Then there exists a sequence of constants $N_G(\ell)$ with $N_G(\ell)\to 0$ as $\ell\to\infty$ such that if $$\tilde\rho: \pi_1^{\text{\'et}}(X_{\bar k}, \bar x)\to GL_n(\overline{\mathbb{Q}_\ell})$$ is semisimple arithmetic with $$\on{Tr}(\rho)\equiv \on{Tr}(\tilde\rho)\bmod \ell^{N_G(\ell)},$$ we have $\rho\otimes \overline{\mathbb{Q}_\ell}\simeq \tilde\rho$. 
\end{theorem}
In other words, Theorem \ref{main-theorem} implies that there is a ball around $\rho\otimes \overline{\mathbb{Q}_\ell}$ in which it is the unique semisimple arithmetic representation of $\pi_1^{\text{\'et}}(X_{\bar k}, \bar x)$.  Theorem \ref{finite-image-thm} implies that the radius of this ball tends to $1$ as $\ell\to\infty$.

In particular, for $\ell\gg 0$, if $L$ is a finite extension of $\mathbb{Q}_\ell$ with residue field $\mathbb{F}_{\ell^r}$, there is a unique semisimple arithmetic lift of any representation $$\pi_1^{\text{\'et}}(X_{\bar k}, \bar x)\to G\to GL_n(\mathbb{F}_{\ell^r})$$ to $GL_n(L)$, namely the obvious one which factors through $G$.  For example, we have the following simple consequence:
\begin{corollary}
There exists $\ell_0(X)\gg 0$, independent of $n$, such that for $\ell>\ell_0$, the unique semisimple arithmetic representation $$\rho: \pi_1^{\text{\'et}}(X_{\bar k})\to GL_n(\mathbb{Z}_\ell),$$ which is trivial mod $\ell$, is the trivial representation. 
\end{corollary}
Here semisimplicity means that $\rho\otimes \mathbb{Q}_\ell$ is semisimple.
\subsection{Overview of the proof} 
\subsubsection{Sketch proof of Theorem \ref{main-theorem}}

The proof proceeds in two steps.  First, we show that any semisimple arithmetic representation into $GL_n(\mathbb{C}_\ell)$ is in fact defined over $\overline{\mathbb{Q}_\ell}$ (Theorem \ref{rigidity-theorem}).  This is the only place in the paper in which Lafforgue's work is used; we require as input from Lafforgue the fact that if $k$ is finite of characteristic different from $\ell$ and $$\rho: \pi_1^{\text{\'et}}(X, \bar x)\to GL_n(\overline{\mathbb{Q}_\ell})$$ is irreducible when restricted to $\pi_1^{\text{\'et}}(X_{\bar k}, \bar x)$, then $\rho\otimes \rho^\vee$ is pure of weight zero.

We reduce to the case of finite fields by a specialization argument.  In this case, we show via a dynamical argument (Corollary \ref{mordell-lang-cor}), that for $c$ as in the Theorem, there exists $k(c)/k$ finite such that any $\rho$ satisfying the condition of Theorem \ref{main-theorem} is invariant under the action of $G_{k(c)}$ (Corollary \ref{M-periodic-arithmetic}).  Now suppose there were infinitely many such semisimple $\rho$.  The condition $\on{Tr}(\rho)\equiv \on{Tr}(\bar\rho)\mod \ell^c$ defines an affinoid subdomain of the rigid generic fiber of the space of pseudorepresentations lifting $\on{Tr}(\bar\rho)$; the infinitude of $\rho$ satisfying the given condition would imply that this subdomain contains infinitely many $G_{k(c)}$-fixed points. Thus the space of $G_{k(c)}$-fixed points would be a \emph{positive-dimensional} rigid space, and hence would have a $\mathbb{C}_\ell$-point not defined over $\overline{\mathbb{Q}_\ell}$.  This contradicts the result of the previous paragraph.

\subsubsection{Sketch proof of Theorem \ref{ball-bound}}
The proof is a variant on the proof of Theorem 1.2 of \cite{litt}, replacing the use of the pro-unipotent fundamental group in that paper with the use of deformation rings. As in the statement of the theorem, let $$\rho: \pi_1^{\text{\'et}}(X_{\bar k}, \bar x)\to GL_n(\overline{\mathbb{Q}_\ell})$$ be a representation which arises from geometry, lifting a geometrically irreducible residual representation $$\bar\rho: \pi_1^{\text{\'et}}(X_{\bar k}, \bar x)\to GL_n(\mathbb{F}_{\ell^r}).$$ Let $S_\rho$ be the deformation ring of $\rho$ and $R_{\bar\rho}$ the deformation ring of $\bar\rho$.  

We define a weight filtration $W^\bullet$ on $S_\rho$ and, for $\alpha\in \mathbb{Z}_\ell^\times$ sufficiently close to $1$, we construct elements $\sigma_\alpha\in G_k$ which act on $\on{gr}^{-i}_WS_\rho$ via $\alpha^i\cdot\on{Id}$ (Theorem \ref{pi1-quasiscalar}).  Using this analysis, we construct a $G_k$-stable closed ball $U$ in rigid generic fiber $E_{\bar\rho}$ of $R_{\bar\rho}$, containing the point of $E_{\bar\rho}$ corresponding to $\rho$, such that the span of the $\sigma_\alpha$ eigenvectors in $\mathscr{O}_U$ is dense.

The $\sigma_\alpha$-eigenvectors in $\mathscr{O}_U$ are convergent power-series vanishing at the point of $E_{\bar\rho}$ corresponding to $\rho$; we estimate their coefficients in terms of $\alpha$.  Using this estimate, we may estimate the radius of a ball $U'$ in which $\rho$ is the unique common zero of these $\sigma_\alpha$-eigenvectors, and hence the unique $\sigma_\alpha$-periodic point of $U'$. We thus conclude that it is the unique arithmetic representation contained in $U'$.
\subsection{Comparison to previous work}
We first discuss predecessors to Theorem \ref{main-theorem} and Corollary \ref{main-corollary}.  Deligne showed \cite{deligne} that if $X$ is a normal complex algebraic variety, there are finitely many representations $$\pi_1^{\text{top}}(X, x)\to GL_n(\mathbb{Q})$$ underlying a polarizable variation of Hodge structure, and hence finitely many such representations which arise from geometry; Corollary \ref{main-corollary} is analogous, but replaces $\mathbb{Q}$ with an $\ell$-adic field.  

Work of Deligne, Drinfel'd, and Lafforgue implies (via automorphic methods) that if $X$ is a variety over a finite field $\mathbb{F}_q$, the set of semisimple $\overline{\mathbb{Q}_\ell}$-representations of its Weil group $W(X)$, with \emph{bounded wild ramification at infinity}, is finite up to twist by characters of $W(\mathbb{F}_q)$ \cite[Theorem 2.1]{esnault-kerz}. We expect that one could deduce Theorem \ref{main-theorem} and Corollary \ref{main-corollary} from this result, using the dynamical methods of Section \ref{dynamics-section} of this paper. Unlike the proof of that result, our proof avoids direct use of automorphic techniques, though it does rely on the work of Lafforgue.

Theorems \ref{ball-bound} and \ref{finite-image-thm} are weak variants of the Frey-Mazur conjecture for function fields; the proofs are a (somewhat involved) variant of the proof of \cite[Theorem 1.2]{litt}. There has been much recent work on the function field Frey-Mazur conjecture for representations arising from families of elliptic curves; see e.g.~\cite{bakker-tsimerman}. Theorem \ref{finite-image-thm} also has complex-analytic analogues (with rather different uniformities) in the case of monodromy representations arising from families of Abelian varieties, in e.g.~\cite{nadel, hwang-to}.

\subsection{Acknowledgments}
To be added after the referee process is complete.
\section{Preliminaries on deformation rings}\label{deformation-ring-preliminaries}
We now begin preparations for the proof of Theorem \ref{main-theorem}, which will proceed by analyzing the dynamics of the Galois action on framed deformation rings of residual representations of $\pi_1^{\text{\'et}}(X_{\bar k}, \bar x)$, and on moduli of pseudorepresentations. We first recall the definitions of the objects in question.
\subsection{Basics of deformation rings}
Let $\ell$ be a prime. Let $G$ be a profinite group satisfying Mazur's condition ($\Phi_\ell$) \cite{mazur1}:
\begin{equation}\tag{$\Phi_\ell$}
\text{\parbox{.85\textwidth}{For each open subgroup of finite index $H\subset G$, $\on{Hom}_{\text{cont}}(H, \mathbb{F}_\ell)$ is finite.}}
\end{equation}

Let $\bar \rho: G\to GL_n(\mathbb{F}_{\ell^r})$ be a continuous representation.  Write $\Lambda=W(\mathbb{F}_{\ell^r})$ for the Witt ring of $\mathbb{F}_{\ell^r}$, and let $\mathscr{C}_\Lambda$ be the category of local Artinian $\Lambda$-algebras with residue field $\mathbb{F}_{\ell^r}$.  Recall that the framed deformation functor $$D^\square_{\bar \rho}: \mathscr{C}_{\Lambda}\to \on{Sets}$$ assigns to an object $A$ of $\mathscr{C}_{\Lambda}$ the set of $$\{\rho: G\to GL_n(A)\mid \rho\otimes_A\mathbb{F}_{\ell^r} =\bar\rho\}/\sim$$ where here $\rho\simeq \rho'$ if there exists an isomorphism $\psi: \rho \overset{\sim}{\to} \rho'$ so that the diagram 
$$\xymatrix{
\rho\otimes_A \mathbb{F}_{\ell^r} \ar[d]^{\psi\otimes_A \mathbb{F}_{\ell^r}} \ar@{=}[r] & \bar\rho \ar@{=}[d]\\
\rho'\otimes_A \mathbb{F}_{\ell^r} \ar@{=}[r] & \bar\rho
}$$
commutes.  In other words, $D^\square_{\bar\rho}$ parametrizes lifts of $\bar\rho$ with a lift of the standard basis, up to the action of $\ker(GL_n(A)\to GL_n(\mathbb{F}_{\ell^r}))$.  

We let $D_{\bar\rho}: \mathscr{C}_{\Lambda}\to \on{Sets}$ be the deformation functor which assigns to an object $A$ of $\mathscr{C}_{\Lambda}$ the set  $$\{(\rho: G\to GL_n(A), \iota:  \rho\otimes_A\mathbb{F}_{\ell^r} \overset{\sim}{\to}\bar\rho)\}/\sim$$ where here $(\rho, \iota)\sim (\rho', \iota')$ if there exists a commutative diagram 
$$\xymatrix{
\rho\otimes_A \mathbb{F}_{\ell^r} \ar[d]^{\sim} \ar[r]^{\iota} & \bar\rho \ar[d]^{\sim}\\
\rho'\otimes_A \mathbb{F}_{\ell^r} \ar[r]^{\iota'} & \bar\rho.
}$$

There is an evident map $D^{\square}_{\bar\rho}\to D_{\bar\rho}$, given by forgetting the framing.

As $G$ satisfies Mazur's finiteness condition $(\Phi_\ell)$, $D^\square_{\bar \rho}$ is pro-representable by a local Noetherian $\Lambda$-algebra $R^{\square}_{\bar\rho}$ with residue field $\mathbb{F}_{\ell^r}$ \cite{mazur1}.  In general the functor $D_{\bar\rho}$ is not pro-representable, though it is if $\bar\rho$ is absolutely irreducible;  in this case we call the pro-representing object $R_{\bar\rho}$.  The groups $$T_i=H^i_{\text{cont}}(G, \bar\rho\otimes \bar\rho^\vee)$$ form a tangent-obstruction theory for $D_{\bar \rho}$.  

In particular, if $T_2=0$,  $D_{\bar\rho}$ is formally smooth; as the forgetful natural transformation $D^{\square}_{\bar\rho}\to D_{\bar\rho}$ is formally smooth, this implies $D^\square_{\bar\rho}$ is formally smooth as well.  In this situation, $R^\square_{\bar\rho}$ is, by the Cohen structure theorem, non-canonically isomorphic to a power series ring over $\Lambda$; if it exists, $R_{\bar\rho}$ is non-canonically isomorphic to a power series ring over $\Lambda$ as well, with tangent space canonically isomorphic to $T_1$.

Finally, we recall from \cite{taylor, chenevier} the definition of a pseudorepresentation, which formalizes the algebraic properties of the determinant of a representation $\on{det}(\rho): R[[G]]\to R$, where $$\rho: G\to GL_n(R)$$ is a representation.
\begin{defn}[Pseudorepresentations]
Let $A$ be a commutative ring and $R$ a not-necessarily commutative $A$-algebra.  Let $A-\text{alg}$ be the category of commutative $A$-algebras, and let $$\underline{R}: A-\text{alg}\to \on{A}-\text{mod}$$ be the functor $$S\mapsto R\otimes_A S.$$
\begin{enumerate}
\item A polynomial law $D: R\to A$ is a natural transformation $\underline{R}\to \underline{A}$.  If $B$ is a commutative $A$-algebra, we let $D_B: R\otimes_A B\to B$ be the induced map.
\item A polynomial law $D$ is homogeneous of degreee $n$ if $D_B(xb)=b^nD(x)$ for all $b\in B, x\in R\otimes_A B$.
\item A $d$-dimensional pseudorepresentation of $R$ is a homogeneous polynomial law $D: R\to A$ of degree $d$.
\item If $G$ is a group, then a $d$-dimensional $A$-pseudorepresentation of $G$ is a homogeneous pseudorepresentation $D: A[G]\to A$ of degree $d$.
\item If $D$ is a $d$-dimensional $A$-pseudorepresentation of $G$, we define its characteristic polynomial $\chi(g,t)$ by $$\chi(g,t):=D_{A[t]}(t-g)=\sum_{i=0}^d (-1)^i\Lambda_i^D(g)t^{d-i}$$ for $g\in G$.  We define the trace of $D$ to be $\Lambda_1^D$.
\item If $A$ is a topological commutative ring and $G$ is a topological group, we say that a $d$-dimensional pseudorepresentation $D$ of $G$ is \emph{continuous} if $\Lambda_i^D:G\to A$, defined as above, is continuous for all $i$. (See \cite[Section 2.30]{chenevier}.)
\end{enumerate}
\end{defn}
Of course if $$\rho: G\to GL_n(R)$$ is a continuous representation, then $\det\circ\rho$ is a continuous $n$-dimensional pseudorepresentation of $G$.  

We will use throughout that if $R$ is an algebraically closed field of characteristic zero, then a pseudorepresentation is determined uniquely by its trace \cite[Proposition 1.29]{chenevier}, and that $$\rho\mapsto \det\circ\rho$$ gives a bijection between (continuous) semisimple representations into $GL_n(R)$ and (continuous) $n$-dimensional $R$-pseudorepresentations \cite[Theorem 2.12]{chenevier}.  

Given a $d$-dimensional pseudorepresentation $\bar r: G\to \mathbb{F}_{\ell^r}$, we let $$D_{\bar r}^{\text{ps}}: \mathscr{C}_{\Lambda}\to \on{Sets}$$ be the functor which assigns to an Artin $\Lambda$-algebra $A$ with residue field $\mathbb{F}_{\ell^r}$ the set of $d$-dimensional pseudorepresentations $G\to A$ lifting $\bar r$.  Chenevier shows \cite[Proposition 3.3]{chenevier} that $D_{\bar r}^{\text{ps}}$ is pro-representable by a local Noetherian $\Lambda$-algebra $A(\bar r)$.  Given an $n$-dimensional residual representation $\bar\rho$ as above, there is a natural map  $D_{\bar\rho}\to D_{\on{det}\circ \bar\rho}^{\text{ps}}$, given by sending a deformation of $\bar\rho$ to its determinant.

We will also at several places in this text use the rigid-analytic moduli space of pseudorepresentations. Briefly, if $\on{An}$ is the category of rigid-analytic spaces over $\mathbb{Q}_\ell$, and $$E^{\text{an}}: \on{An}\to \on{Sets}$$ is the functor which associates to $X$ the set of $d$-dimensional pseudorepresentations $G\to \mathscr{O}(X)$, Chenevier shows that $E^{\text{an}}$ is represented by a quasi-Stein rigid analytic space, which we will denote $E_d$.  Chenevier shows that $E_d$ is the disjoint union of the rigid generic fibers of the $A(\bar r)$, where $\bar r$ ranges over all residual $d$-dimensional pseudorepresentations.

Finally, if $\bar\rho: G \to GL_n(\mathbb{F}_{\ell^r})$ is an $n$-dimensional representation of $G$, there is for each $c\in \mathbb{R}$ with $1\geq c > 0$ an affinoid subdomain of $E_n$ given by the set of pseudorepresentations $D: G\to \overline{\mathbb{Q}_\ell}$ with $\Lambda_1^D\equiv \on{Tr}(\bar\rho)\bmod \ell^{c}$.   We denote this subdomain by $E_{\bar\rho, c}$.  Let $$E_{\bar\rho}=\bigcup_{1\geq c>0} E_{\bar\rho, c}.$$
\subsection{Galois actions on deformation rings}
Let $X$ be a normal, geometrically connected variety over a finitely generated field $k$ of characteristic different from $\ell$; let $x\in X(k)$ be a rational point.  Choose an algebraic closure $\bar k$ of $k$, and let $\bar x\in X(\bar k)$ be the geometric point of $X$ associated to $x$.  The fact that $x$ is a rational point means that the pair $(X, x)$ has an action by $G_k:=\on{Gal}(\bar k/k)$; hence $G_k$ acts on $\pi_1^{\text{\'et}}(X_{\bar k}, \bar x)$.

Let $$\bar \rho: \pi_1^{\text{\'et}}(X_{\bar k}, \bar x)\to GL_n(\mathbb{F}_{\ell^r})$$ be a continuous representation.  We will now apply the discussion of the previous section to the case $G=\pi_1^{\text{\'et}}(X_{\bar k}, \bar x)$.

As $\ell\not=\on{char}(k)$, $\pi_1^{\text{\'et}}(X_{\bar k}, \bar x)$ satisfies Mazur's finiteness condition $(\Phi_\ell)$ (by the finite-generation of $H^1_{\text{\'et}}(X_{\bar k}', \mathbb{F}_\ell)$, where $X'$ is any finite \'etale cover of $X$), and hence $D^\square_{\bar \rho}$ is pro-representable by a local Noetherian $\Lambda$-algebra $R^{\square}_{\bar\rho}$.  

Now $G_k$ acts on $\pi_1^{\text{\'et}}(X_{\bar k}, \bar x)$, and hence on the space of pseudorepresentations of $\pi_1^{\text{\'et}}(X_{\bar k}, \bar x)$, denoted $E_n$ as above.  Moreover, as $GL_n(\mathbb{F}_{\ell^r})$ is finite, there exists a finite index subgroup $H_1\subset G_k$ so that for each $h\in H_1, \bar\rho^h\simeq \bar\rho$.  Let $k_1=\bar k^{H_1}$, so $H_1=G_{k_1}$.  Then $G_{k_1}$ acts on $D_{\bar\rho}$ via its action on $\pi_1^{\text{\'et}}(X_{\bar k}, \bar x)$, and hence on $R_{\bar\rho}$ if it exists. An identical argument shows that $G_{k_1}$ acts on $E_{\bar\rho, c}$ for each $c$ with $1\geq c>0$, and on $A(\det\circ \bar\rho)$.

\section{Properties of arithmetic representations}
\subsection{Basic properties}
We now establish some basic properties of arithmetic representations.
\begin{prop}\label{basic-properties}
Let $K=\overline{\mathbb{Q}_\ell}$ or $K=\mathbb{C}_\ell:=\widehat{\overline{\mathbb{Q}_\ell}}$.  Let $$\rho: \pi_1^{\text{\'et}}(X_{\bar k}, \bar x)\to GL_n(K)$$ be a continuous, semisimple representation.  Then the following are equivalent:
\begin{enumerate}
\item There exists a finite extension $k'$ of $k$ and a continuous representation $$\tilde \rho: \pi_1^{\text{\'et}}(X_{k'}, \bar x)\to GL_n(K)$$ such that $\tilde\rho|_{\pi_1^{\text{\'et}}(X_{\bar k}, \bar x)}\simeq \rho.$
\item $\rho$ is arithmetic.
\item There exists an open subgroup $H\subset G_k$ such that for all $h\in H, \rho^h\simeq \rho$.  
\item There exists an open subgroup $H\subset G_k$ such that for all $h\in H$, $\on{Tr}(\rho^h(g))=\on{Tr}(\rho(g))$.
\end{enumerate}
\end{prop}
\begin{proof}
Clearly $(1)\implies(2)$ and $(3)\implies (4)$.

We now show that $(2)\implies (3)$.  By definition, there exists a finite extension $k'$ of $k$ and a continuous representation $$\tilde \rho: \pi_1^{\text{\'et}}(X_{k'}, \bar x)\to GL_n(K)$$ such that $\rho$ is a subquotient of $\tilde\rho|_{\pi_1^{\text{\'et}}(X_{\bar k}, \bar x)}$.  Let $\tilde S$ be the (finite) set of isomorphism classes of irreducible representations of $\pi_1^{\text{\'et}}(X_{\bar k}, \bar x)$ appearing as subquotients of $\tilde \rho|_{\pi_1^{\text{\'et}}(X_{\bar k}, \bar x)}$, and let $S\subset \tilde S$ be the set of irreducible representations appearing as subquotients of $\rho$.  $G_{k'}$ permutes $\tilde S$ (acting via its outer action on $\pi_1^{\text{\'et}}(X_{\bar k}, \bar x)$), and thus we may set $H$ to be the finite-index subgroup of $G_{k'}$ fixing each element of $S$.  As $\rho$ is semisimple, this implies that (3) holds.

The fact that $(4)\implies (3)$ is immediate from \cite[Theorem 1]{taylor}.

Finally, we show $(3)\implies (1)$.  We may immediately reduce to the case that $\rho$ is irreducible.  For each $h\in H$, choose an isomorphism $$\gamma_h: \rho^h\overset{\sim}{\to}\rho$$ --- as $\rho$ is irreducible, $\gamma_h$ is well-defined up to scaling, by Schur's lemma.  Thus the assignment $$\gamma: H\to PGL_n(K)$$ $$h\mapsto [\gamma_h]$$ is a well-defined (continuous) homomorphism.  We claim that $\gamma$ lifts to an honest representation $H'\to GL_n(K)$, where $H'$ is an open subgroup of $H$.  Indeed, the obstruction to lifting $\gamma$ from a projective representation to an honest representation is a class $\alpha\in H^2(H, \mathbb{Q}/\mathbb{Z})$.  Choose $H'$ such that $\on{Res}^{H}_{H'}(\alpha)=0$.  Let $\tilde\gamma: H' \to GL_n(K)$ be a choice of lift.

Now let $k'=\bar k^{H'}$ be the fixed field of $H'\subset G_k$, so that $H'=G_{k'}$.  Without loss of generality, we may assume that the exact sequence $$1\to \pi_1^{\text{\'et}}(X_{\bar k}, \bar x)\to \pi_1^{\text{\'et}}(X_{k'}, \bar x)\to G_{k'}\to 1$$ splits (if not, replace $k'$ with a finite extension so that $X$ has a rational point, giving such a splitting), so that $$\pi_1^{\text{\'et}}(X_{k'}, \bar x)\simeq \pi_1^{\text{\'et}}(X_{\bar k}, \bar x) \rtimes G_{k'}.$$  Now we may set $$\tilde \rho = \rho \rtimes \tilde \gamma.$$
\end{proof}
\begin{remark}
A similar argument appears in \cite[Proposition 4.6]{esnault1} and \cite[Proposition 3.1]{esnault2}.
\end{remark}
\begin{corollary}
Let $K=\overline{\mathbb{Q}_\ell}$ or $\mathbb{C}_\ell$. Let $\rho:\pi_1^{\text{\'et}}(X_{\bar k}, \bar x)\to GL_n(K)$ be arithmetic.  Then its semisimplification $\rho^{\text{ss}}$ is arithmetic as well.
\end{corollary}
\begin{proof}
Let $S$ be the set of irreducible subquotients of $\rho$.  The arithmeticity of $\rho$ implies that there exists an open subgroup of $G_k$ which stabilizes $S$; hence $\rho^{\text{ss}}$ is arithmetic by Proposition \ref{basic-properties}(4).
\end{proof}
Motivated by Proposition \ref{basic-properties}, we make the following definitions.
\begin{defn}[Field of moduli and field of definition]\label{field-of-defn}
Let $\rho: \pi_1^{\text{\'et}}(X_{\bar k}, \bar x)\to GL_n(\overline{\mathbb{Q}_\ell})$ be a semisimple arithmetic representation.
\begin{enumerate}
\item Let $H\subset G_k$ be the stabilizer of $\rho$; by Proposition \ref{basic-properties}(3), $H$ is open.  We say that the fixed field of $H$, $\bar k^H$, is the \emph{field of moduli} of $\rho$.
\item If $k'$ is such that there exists $\tilde\rho: \pi_1^{\text{\'et}}(X_{k'}, \bar x)\to GL_n(\overline{\mathbb{Q}_\ell})$ such that $\tilde\rho|_{\pi_1^{\text{\'et}}(X_{\bar k}, \bar x)}\simeq \rho$, we say that $k$ is a \emph{field of definition} of $\rho$.
\end{enumerate}
\end{defn}

If $k$ is a number field, and $\rho$ is irreducible, the field of moduli equals the field of definition:
\begin{theorem}
Suppose $k$ is a number field, $X$ a geometrically connected $k$-variety, and $x\in X(k)$ is a rational point.  Let $\bar k$ be an algebraic closure of $k$ and $\bar x$ the geometric point of $X$ associated to $x$. Let $$\rho: \pi_1^{\text{\'et}}(X_{\bar k}, \bar x)\to GL_n(\overline{\mathbb{Q}_\ell})$$ be a continuous irreducible representation. Then the following are equivalent:
\begin{enumerate}
\item For each $g\in G_k:=\on{Gal}(\bar k/k)$, $\rho^g\simeq \rho$.
\item There exists a representation $\tilde\rho: \pi_1^{\text{\'et}}(X, \bar x)\to GL_n(\overline{\mathbb{Q}_\ell})$ such that $$\rho\simeq \tilde\rho|_{\pi_1^{\text{\'et}}(X_{\bar k}, \bar x)}.$$
\end{enumerate} 
\end{theorem}
\begin{proof}
Clearly $(2)\implies (1)$.  We now show that $(1)\implies (2)$.  As in the proof of Proposition \ref{basic-properties}, we obtain a projective representation $$\gamma: G_k\to PGL_n(\overline{\mathbb{Q}_\ell});$$ as the sequence $$1\to \pi_1^{\text{\'et}}(X_{\bar k}, \bar x)\to \pi_1^{\text{\'et}}(X, \bar x)\to G_{k}\to 1$$ splits (because $x$ is a rational point), it is enough to lift $\gamma$ to an honest representation $$\tilde\gamma: G_k\to GL_n(\overline{\mathbb{Q}_\ell}).$$ But such a lift exists by a result of Tate (see \cite[Theorem 2.1.1]{patrikis}).
\end{proof}
\begin{defn}\label{arises-from-geometry-defn}
A representation $$\rho: \pi_1^{\text{\'et}}(X_{\bar k}, \bar x)\to GL_n(L)$$ \emph{arises from geometry} if there exists an algebraically closed field $F$ with $k\subset F$, and a smooth proper map $f: Y\to X_F$, where $Y$ is an $F$-variety, such that $\rho$ is a subquotient of the monodromy representation $$\pi_1^{\text{\'et}}(X_{\bar k}, \bar x)\to GL((R^if_*\underline{L})_{\bar x}).$$
\end{defn}
\begin{prop}\label{comes-from-geometry}
Let $$\rho: \pi_1^{\text{\'et}}(X_{\bar k}, \bar x)\to GL_n(L)$$ be a representation which arises from geometry.  Then $\rho$ is arithmetic.
\end{prop}
\begin{proof}
This is a standard spreading-out and specialization argument, which we include for the reader's convenience.

Let $F$ be an algebraically closed extension of $k$ such that there exists an $F$-variety $Y$ and a smooth proper map $f: Y\to X_F$ with $\rho$ appearing as a subquotient of the monodromy representation on $(R^if_*\underline{L})_{\bar x}$.  There exists a finitely-generated $k$-algebra $R\subset F$, an $R$-scheme $\mathscr{Y}$, an isomorphism $\iota: Y\overset{\sim}{\to} \mathscr{Y}_F$, and a smooth proper morphism $g: \mathscr{Y}\to X_R$ such that the diagram
$$\xymatrix{
Y \ar[r]^\iota_\sim \ar[d]^f & \mathscr{Y}_F \ar[d]^{g_F}\\
X_F \ar@{=}[r] & X_F
}$$
commutes.  Now specializing to a closed point $z$ of $\on{Spec}(R)$, we find that $(R^ig_{z*}\underline{L})$ is a lisse sheaf with the same monodromy representation as $R^if_*\underline{L}$ on $X_{k(z)}$; hence the representation in question is arithmetic as desired.
\end{proof}
\subsection{Rigidity}
We now prove a rigidity statement for arithmetic representations; this, along with some results in $\ell$-adic dynamics proved in Section \ref{dynamics-section}, will be the main ingredient in the proof of Theorem \ref{main-theorem}.
\begin{lemma}\label{rigidity-lemma}
Let $k=\mathbb{F}_q$ be a finite field with algebraic closure $\bar k$, $C/k$ a smooth affine curve, and $\ell$  a prime not dividing $\on{char}(k)$.  Let $x\in C(k)$ be a rational point, and $\bar x$ the associated geometric point of $C$.  Let $$\rho: \pi_1^{\text{\'et}}(C_{\bar k}, \bar x)\to GL_n(\overline{\mathbb{Q}_\ell})$$ be a continuous irreducible arithmetic representation.  Let $A$ be a local Artinian $\overline{\mathbb{Q}_\ell}$-algebra with residue field $\overline{\mathbb{Q}_\ell}$ and let $$\tilde\rho: \pi_1^{\text{\'et}}(C_{\bar k}, \bar x)\to GL_n(A)$$ be a deformation of $\rho$ such that $\on{Tr}(\tilde\rho^{\phi_x})= \on{Tr}(\tilde\rho)$, where $\phi_x$ is the Frobenius at $x$.  Then $\tilde\rho\simeq \rho\otimes_{\overline{\mathbb{Q}_\ell}} A$.
\end{lemma}
\begin{proof}
Let $W(k):=\mathbb{Z}\cdot \on{Frob}\subset \on{Gal}(\bar k/k)$ be the subgroup of the Galois group of $k$ generated by Frobenius.  Let the Weil group of $C$, denoted $W(C)$, be the fiber product
$$\xymatrix{
W(C):=W(k)\times_{\on{Gal}(\bar k/k)} \pi_1^{\text{\'et}}(C, \bar x) \ar[r] \ar[d] & W(k) \ar@{^(->}[d] \\
\pi_1^{\text{\'et}}(C, \bar x) \ar[r] & \on{Gal}(\bar k/k).
}$$

Observe that as $\rho$ is arithmetic, there exists an isomorphism $\gamma: \rho\overset{\sim}{\to} \rho^{\phi_x}$; as $\rho$ is irreducible, $\gamma$ is well-defined up to scaling.  Our choice of $\gamma$ extends $\rho$ to a $W(C)$ representation, well-defined up to a character of $W(k)$; thus $\rho\otimes \rho^\vee$ is well-defined as a representation of $W(C)$, and has weight zero by work of Lafforgue \cite[Corollaire VII.8]{lafforgue}.

Let $\mathfrak{m}_A$ be the maximal ideal of $A$, and let $I\subset A$ be a non-zero ideal with $\mathfrak{m}_A\cdot I=0$, so that $$0\to I\to A\to B\to 0$$ is a small extension of local Artinian $\overline{\mathbb{Q}_\ell}$-algebras with residue field $\overline{\mathbb{Q}_\ell}$.  Let $$\tilde\rho:\pi_1^{\text{\'et}}(C_{\bar k}, \bar x)\to GL_n(A)$$ be a deformation of $\rho$ satisyfing the hypotheses of the theorem, i.e.~$\tilde\rho\otimes_B k\simeq \rho$ and $\on{Tr}(\tilde\rho^{\phi_x})=\on{Tr}(\tilde \rho)$.  By induction on the length of $A$ we may assume that $\tilde\rho\otimes_A B\simeq \rho\otimes_{\overline{\mathbb{Q}_\ell}} B$.

Now $\tilde\rho$ and $\rho\otimes_{ \overline{\mathbb{Q}_\ell}} A$ are deformations of $\rho\otimes_{\overline{\mathbb{Q}_\ell}}B$; the space of deformations of $\rho\otimes_{\overline{\mathbb{Q}_\ell}} B$ is a torsor for $H^1(\pi_1^{\text{\'et}}(C_{\bar k}, \bar x), \rho\otimes \rho^\vee)\otimes I$.  Thus $[\tilde\rho]-[\rho\otimes_{\overline{\mathbb{Q}_\ell}} A]$ gives a well-defined class in $H^1(\pi_1^{\text{\'et}}(C_{\bar k}, \bar x), \rho\otimes \rho^\vee)\otimes I$.

By \cite[Th\'{e}or\`{e}me 1]{carayol}, $\tilde\rho^{\phi_x}\simeq \tilde\rho$; clearly $\rho\otimes_{\overline{\mathbb{Q}_\ell}} A$ is also $\phi_x$-invariant.  Hence $$[\tilde\rho]-[\rho\otimes_{\overline{\mathbb{Q}_\ell}} A]\in (H^1(\pi_1^{\text{\'et}}(C_{\bar k},\bar x), \rho\otimes \rho^\vee)\otimes I)^{G_k}.$$

But recall that $\rho\otimes \rho^\vee$ has weight zero; thus by Weil II \cite{weilii}, $H^1(\pi_1^{\text{\'et}}(C_{\bar k},\bar x), \rho\otimes \rho^\vee)$ has weights in $\{1, 2\}$. In particular $(H^1(\pi_1^{\text{\'et}}(C_{\bar k},\bar x), \rho\otimes \rho^\vee)\otimes I)^{G_k}=0$.  So $\tilde\rho\simeq \rho\otimes_{\overline{\mathbb{Q}_\ell}} A$ as desired.
\end{proof}
\begin{remark}
The proof of Lemma \ref{rigidity-lemma} is the only place in this paper where the work of Lafforgue is used, where we need above that $\rho\otimes\rho^\vee$ has weight zero.
\end{remark}

We now deduce that all semisimple arithmetic representations into $GL_n(\mathbb{C}_\ell)$ are defined over $\overline{\mathbb{Q}}_\ell$.  
\begin{theorem}\label{rigidity-theorem}
Let $X$ be a normal, geometrically connected curve over a finite field $k$, and let $\ell$ be a prime different from the characteristic of $k$.  Then any semisimple arithmetic representation $$\rho: \pi_1^{\text{\'et}}(X_{\bar k}, \bar x)\to GL_n(\mathbb{C}_\ell)$$ with $$\on{Tr}(\rho)\bmod \mathfrak{m}_{\mathscr{O}_{\mathbb{C}_\ell}}\subset \mathbb{F}_{\ell^r}$$ for some $r$ is in fact defined over $\overline{\mathbb{Q}_\ell}$, i.e.~there exists a representation $$\tilde\rho: \pi_1^{\text{\'et}}(X_{\bar k}, \bar x)\to GL_n(\overline{\mathbb{Q}_\ell})$$ such that $\tilde\rho\otimes_{\overline{\mathbb{Q}_\ell}}\mathbb{C}_\ell\simeq \rho$.
\end{theorem}
\begin{proof}
We first claim it suffices to prove the theorem for $\rho$ irreducible. Indeed $\rho$ is a direct sum of irreducible constituents; if each is defined over $\overline{\mathbb{Q}_\ell}$, then their direct sum is as well.  Hence we assume $\rho$ is irreducible.  Moreover, we may assume $X$ is affine,  by deleting a point. 

Let $$\rho_{int}: \pi_1^{\text{\'et}}(X_{\bar k}, \bar x)\to GL_n(\mathscr{O}_{\mathbb{C}_\ell})$$ be an integral model of $\rho$, and let $$\bar\rho: \pi_1^{\text{\'et}}(X_{\bar k}, \bar x)\to GL_n(\overline{\mathbb{F}_\ell})$$ be the residual representation.  By the assumption on $\on{Tr}(\rho)$, $\bar\rho$ factors through $GL_n(\mathbb{F}_{\ell^{r'}})$ for some $r'$; we rename $r'$ as $r$ and abuse notation to refer to the representation $$\pi_1^{\text{\'et}}(X_{\bar k}, \bar x)\to GL_n(\mathbb{F}_{\ell^r})$$ as $\bar\rho$ as well. Let $R^\square_{\bar\rho}$ be the framed deformation ring of $\bar\rho$, as in Section \ref{deformation-ring-preliminaries}.  The representation $\rho$ is classified by a map $$\tilde f_\rho: R^\square_{\bar\rho}\to \mathbb{C}_\ell;$$ let $S$ be the image of $\tilde f_\rho$ inside of $\mathbb{C}_\ell$.  The map $R^\square_{\bar \rho}\to S$ gives a representation $$\rho_S: \pi_1^{\text{\'et}}(X_{\bar k}, \bar x)\to GL_n(S)$$ such that $\rho_S\otimes_S \mathbb{C}_\ell\simeq \rho$.  That is, $\rho_S$ is a model of $\rho$ over a Noetherian local $W(\mathbb{F}_{\ell^r})$-algebra.

There exists a finite-index subgroup $H\subset \on{Gal}(\bar k/k)$ such that for each $h\in H$, $\on{Tr}(\rho_S^h)=\on{Tr}(\rho_S)$, as the same is true for $\rho$.  Now by e.g.~\cite[Proposition G]{chenevier}, there exists a Zariski-open subset $U^{\text{irr}}\subset \on{Spec}(S[1/\ell])$ such that for each point $z: \on{Spec}(\overline{\mathbb{Q}_\ell})\to U^{\text{irr}}$ of $U^{\text{irr}}$, the representation $\rho_z:=\rho_S\otimes_S \overline{\mathbb{Q}_\ell}$ is irreducible.  Choose such a point, and let $\widehat S$ be the completion of $S$ at $z$; as $S$ is an integral domain, $S$ injects into $\widehat S$.  But by Lemma \ref{rigidity-lemma}, $\on{Tr}(\rho_S\otimes_S \widehat S)(g)\subset \overline{\mathbb{Q}_\ell}$ for each $g\in \pi_1^{\text{\'et}}(X_{\bar k}, \bar x)$.  Hence the same is true for $\rho$, and thus $\rho$ is defined over $\overline{\mathbb{Q}_\ell}$ by \cite[Theorem 1]{taylor}.
\end{proof}
\section{Dynamics of deformation rings}\label{dynamics-section}
\subsection{The local dynamical Mordell-Lang conjecture}
We now prove some general facts about continuous pro-finite group actions on $\Lambda[[x_1, \cdots, x_N]]$, which will be used in the proof of Theorem \ref{main-theorem} and related results.  The main technical tool is a uniform local version of the dynamical Mordell-Lang conjecture (Lemma \ref{mordell-lang}); it is surely well-known to experts, but we include a proof as we were unable to find a version with the required uniformities in the literature.

As above, let $\Lambda=W(\mathbb{F}_{\ell^r})$; endow $\Lambda$ with the usual absolute value $|\cdot|_\ell$, so that $|\ell|_\ell=1/\ell$.  Let $R=\Lambda[[x_1, \cdots, x_N]]$.  Let $\Lambda\langle x_1, \cdots, x_N\rangle$ be the Tate algebra on $N$ variables, i.e.~$\Lambda\langle x_1, \cdots, x_N\rangle\subset R$ is the set of power series $$\sum_{I\in \mathbb{Z}_{\geq 0}^N} a_Ix^I$$ with $|a_I|_\ell\to 0$ as $|I|=\sum_{j=1}^N i_j\to \infty$.

We now recall the ``$\ell$-adic analytic arc lemma," in a form due to Poonen \cite{poonen}, building on results of Bell, Ghioca, and Tucker \cite{bell-ghioca-tucker}:
\begin{lemma}[$\ell$-adic analytic arc lemma, {\cite[Theorem 1]{poonen}}]\label{poonen-lemma}
Let $f\in \Lambda\langle x_1, \cdots, x_N\rangle^N$ satisfy $f(\mathbf{x})=\mathbf{x} \bmod \ell^c$ for some $c>\frac{1}{\ell-1}$.  Then there exists $g\in \Lambda\langle x_1, \cdots, x_N, n\rangle^N$ with $g(\mathbf{x}, m)=f^m(\mathbf{x})$ for each $m\in \mathbb{Z}_{\geq 0}$.
\end{lemma}
In other words, iterates of analytic self-maps of the closed unit ball which are sufficiently close to the identity may be $\ell$-adically interpolated.

Using Lemma \ref{poonen-lemma}, we will prove a uniform local version of the dynamical Mordell-Lang conjecture.  Let $R=\Lambda[[x_1, \cdots, x_n]]$, and let $U(R)$ be the rigid generic fiber of $R$; this is the open unit ball.  If $L$ is a $\ell$-adic field, an $L$-point of $U(R)$ is a continuous ring homomorphism $R\to L$.  For $1\geq c>0$ a positive real number, let $U_c(R)\subset U(R)$ be the closed ball of radius $\ell^{-c}$ around any $\Lambda[1/\ell]$-point of $U(R)$.  Note that $U_c(R)$ is independent of the choice of $\Lambda[1/\ell]$-point, by the ultrametric inequality.  In particular, if $$\varphi: \Lambda[[x_1, \cdots, x_N]]\overset{\sim}{\to} \Lambda[[x_1, \cdots, x_N]]$$ is a continuous automorphism, then $\varphi$ restricts to an automorphism of $U_c(R)$ for all $c\leq 1$.  An $L$-point $z: R\to L$ of $U(R)$ lands in $U_c(R)$ if $|z(x_i)|_\ell \leq \ell^{-c}$ for all $i$.  Each $U_c(R)$ is an affinoid subdomain of $U(R)$.

If $\mathscr{I}\subset \mathscr{O}_{\mathbb{C}_\ell}[[x_1, \cdots, x_N]]$ is an ideal, we let $V(\mathscr{I})\subset U(R)$ be the set $$V(\mathscr{I}):=\{z\mid z(\mathscr{I})=\{0\}\}.$$  We call a set which arises this way a \emph{closed analytic subset} of $U(R)$.  The dynamical Mordell-Lang conjecture asks for a characterization of the set of $m\in \mathbb{Z}_{\geq 0}$ such that $\varphi^m(z)\subset V(\mathscr{I})$, where $\varphi$ is an analytic automorphism of $U(R)$ and $z\in U(R)$ and $\mathscr{I}$ --- it asserts that this set is semilinear.  
\begin{defn}[Semilinear sets]
A set $A\subset \mathbb{Z}_{\geq 0}$ is \emph{semilinear with period $M$} if it is the union of a finite set with finitely many residue classes modulo $M$.  That is, it is a union $$S\cup \{a_i+jM\mid a_i\subset T, j\in \mathbb{Z}_{\geq 0}\}$$ where $S\subset \mathbb{Z}_{\geq 0}$ is finite and $T\subset \{0,1, \cdots, M-1\}$.
\end{defn}

\begin{lemma}[Uniform local dynamical Mordell-Lang conjecture]\label{mordell-lang}
Let $c\in \mathbb{R}$ be a real number with $1>c>0$.  Then there exists $M=M(c, \ell^r, N)$ with the following property: if $$\varphi: R=\Lambda[[x_1, \cdots, x_N]]\overset{\sim}{\to} \Lambda[[x_1, \cdots, x_N]]$$ is a continuous automorphism, $z\in U_c(R)$, and $\mathscr{I}\subset \mathscr{O}_{\mathbb{C}_\ell}[[x_1, \cdots, x_N]]$ is a proper ideal, then the set  $$\{m\in \mathbb{Z}_{\geq 0}\mid \varphi^m(z)\in V(\mathscr{I})\}$$ is semilinear with period $M$.
\end{lemma}
\begin{proof}
Let $c'$ be a rational number with $c>c'>0$, so that $U_{c'}(R)$ contains $U_c(R)$; it suffices to prove the theorem with $c$ replaced by $c'$.  Choose a (possibly ramified) finite extension $\Lambda'$ of $\Lambda$ so that there exists $\varpi\in \Lambda'$ with $|\varpi|_\ell=\ell^{-c'}$.  There exists $M_1$ depending only on $c', N, \ell^r$ such that $\varphi^{M_1}(\mathbf{0})=\mathbf{0}\bmod \varpi$.  Let $$\tilde\varphi(\mathbf{x})=\frac{1}{\varpi}\varphi^{M_1}(\varpi\cdot \mathbf{x}).$$  Note that $\tilde\varphi$ lies in $\Lambda'\langle x_1, \cdots, x_N\rangle$.  Then there exists $M_2>0$ depending only on $c', \ell^r,N$ such that $\tilde\varphi^{M_2}$ satisfies the hypotheses of Lemma \ref{poonen-lemma}; let $\vartheta\in \Lambda'\langle x_1, \cdots, x_N, n\rangle$ be such that $\vartheta(\mathbf{x}, m)=\tilde\varphi^{M_2m}(\mathbf{x})$ for each $m\in \mathbb{Z}_{\geq 0}$, and let $M=M_1M_2$.

Let $f\in \mathscr{I}$.  Without loss of generality $|f(\mathbf{0})|_\ell\leq \ell^{-c'}$, as otherwise $V(\mathscr{I})\cap U_{c'}(R)=\emptyset$; in particular, $\vartheta(\frac{1}{\varpi}f, m)$ converges for any $m\in \Lambda'$.  We have 
\begin{align*}
\{m\in \mathbb{Z}_{\geq 0}\mid \varphi^m(z)\in V(f)\} &=\{m\in \mathbb{Z}_{\geq 0}\mid z\circ \varphi^m(f)=0\}\\
&= \bigcup_{j=0}^{M-1} \{j+Mm\in \mathbb{Z}_{\geq 0}\mid z\circ \varphi^{j+Mm}(f)=0, m\in \mathbb{Z}_{\geq 0}\}\\
&= \{j+Mm\in \mathbb{Z}_{\geq 0}\mid z\circ \varphi^j \circ(\varpi\cdot \vartheta(\frac{1}{\varpi}\cdot f, m))=0, m\in \mathbb{Z}_{\geq 0}\}.
\end{align*}
But the functions $m\mapsto z\circ \varphi^j \circ(\varpi\cdot \vartheta(\frac{1}{\varpi}\cdot f, m)), j=0, \cdots M-1$ are $\ell$-adic analytic in $m$; hence they either have finitely many zeroes in $\mathbb{Z}_{\geq 0}$ or are identically zero. 

Thus for each $f\in\mathscr{I}$, the set of $m\in \mathbb{Z}_{\geq 0}$ such that $\varphi^m(z)\in V(f)$ is semilinear with period $M$. We have  $$\{m\in \mathbb{Z}_{\geq 0}\mid \varphi^m(z)\in V(\mathscr{I})\}=\bigcap_{f\in\mathscr{I}} \{m\in \mathbb{Z}_{\geq 0}\mid \varphi^m(z)\in V(f)\}.$$ But an arbitrary intersection of sets which are semilinear with period $M$ is semilinear with period $M$, from which we may conclude the result.
\end{proof}
\begin{remark}
Note that the constant $M$ of Lemma \ref{mordell-lang} is independent of $\varphi$.
\end{remark}
\begin{corollary}\label{mordell-lang-cor}
Let $R$ be a local Noetherian $\Lambda$-algebra.  Then for any affinoid subdomain $U$ of the rigid generic fiber of $R$, there exists an integer $M=M(U)$ such that: If $\varphi: R\overset{\sim}{\to} R$ is an automorphism so that $U$ is $\varphi$-stable, any $\varphi$-periodic point $z$ of $U$ satisfies $\varphi^M(z)=z$.
\end{corollary}
\begin{proof}
Write $R=S/\mathscr{J}$, where $S=\Lambda[[x_1, \cdots, x_N]]$ and $\mathscr{J}\subset S$ is an ideal, so the rigid generic fiber of $R$ is a closed analytic subset of the open unit ball.  We may lift $\varphi$ to an automorphism $\tilde\varphi$ of $\Lambda[[x_1, \cdots, x_N]]$.  Now by quasicompactness of affinoids, $U$ is contained in $U_c(S)$ for some $c$ with $1\geq c>0$.  Now let $z: R\to L$ be a $\varphi$-periodic point of $U$; it is a $\tilde\varphi$-periodic point of $U_c(S)$.  Let $\mathscr{I}\subset S$ be the ideal cutting out $z$, i.e.~$\mathscr{I}=\ker(S\to R\overset{z}{\to} L)$.  Now the result follows from Lemma \ref{mordell-lang}, applied to $z, \mathscr{I}$; note that the integer $M$ coming from Lemma \ref{mordell-lang} only depends on $U$, and not on $\varphi, z,$ etc.
\end{proof}
We now apply this result to the case where $R$ is a deformation ring. Recall from Section \ref{deformation-ring-preliminaries} that if $\bar\rho$ is a residual representation of $\pi_1^{\text{\'et}}(X_{\bar k}, \bar x)$ and $\det\circ \bar\rho$ the associated pseudorepresentation, we denoted the deformation ring of $\det\circ\bar\rho$ by $A(\det\circ\bar\rho)$.  For $c$ with $1\geq c>0$, we let $E_{\bar\rho, c}$ be the affinoid of the rigid generic fiber of $A(\det\circ\bar\rho)$ consisting of pseudorepresentations with trace equal to the trace of $\bar\rho$ mod $\ell^c$.
\begin{corollary}\label{M-periodic-arithmetic}
Let $X$ be a smooth, geometrically connected curve over a finite field $k$ of characteristic $p,$ and let $\ell$ be a prime different from $p$; let $\on{Frob}\in G_k$ be the Frobenius element.  Let $$\bar\rho: \pi_1^{\text{\'et}}(X_{\bar k}, \bar x)\to GL_n(\mathbb{F}_{\ell^r})$$ be a continuous representation such that $\bar\rho^{\on{Frob}}\simeq \bar\rho$, so $\on{Frob}$ acts on $E_{\bar\rho}$.  Then for any $c$ with $1\geq c>0$, there exists $M\in\mathbb{Z}_{>0}$ such that any $\on{Frob}$-periodic point of $E_{\bar\rho, c}$ is fixed by $\on{Frob}^M$. 
\end{corollary}
\begin{proof}
By the assumption that $\bar\rho^{\on{Frob}}\simeq \bar\rho$, $\on{Frob}$ acts on $A(\det\circ \bar\rho)$; now we are in precisely the situation of Corollary \ref{mordell-lang-cor}, setting $R=A(\det\circ\bar\rho), \varphi=\on{Frob},$ and $U=E_{\bar\rho, c}$.
\end{proof}
\begin{remark}
Unwinding the proof, we used Lemma \ref{poonen-lemma} to interpolate the action of the powers of $\on{Frob}$ on $A(\det\circ\bar\rho)$. We knew a priori that there was a continuous interpolation (coming from the action of $G_k$ on $A(\det\circ\bar\rho)$) --- the input of Lemma \ref{poonen-lemma} is required to see that this action is locally analytic.
\end{remark}
\subsection{Finiteness}
We now prove Theorem \ref{main-theorem} and Corollary \ref{main-corollary}.
\begin{proof}[Proof of Theorem \ref{main-theorem}]
The proof proceeds by reduction to the case where $k$ is a finite field.

\emph{Step 1.} We first prove the theorem under the assumption that $k$ is finite. Let $$\bar\rho: \pi_1^{\text{\'et}}(X_{\bar k}, \bar x)\to GL_n(\mathbb{F}_{\ell^r})$$ be a residual representation as in the statement of the theorem; let $A(\det\circ\bar\rho)$ be the deformation ring of the pseudorepresentation corresponding to $\bar\rho$, defined in Section \ref{deformation-ring-preliminaries}.  Let $E_{\bar\rho, c}$ be the affinoid of the rigid generic fiber of $A(\det\circ\bar\rho)$, also defined in Section \ref{deformation-ring-preliminaries}. After extending $k$, we have that $\bar\rho$ is $G_k$-fixed (as it has finite image) and hence that $G_k$ acts on $A(\det\circ\bar\rho)$ and $E_{\bar\rho, c}$; arithmetic representations correspond to $G_k$-periodic $\mathbb{C}_\ell$-points of $E_{\bar\rho, c}$ (after Theorem \ref{rigidity-theorem}, we may assume they are $\overline{\mathbb{Q}_\ell}$-points).  It suffices to show that there are finitely many such points of $E_{\bar\rho, c}$.

But we are in the situation of Corollary \ref{M-periodic-arithmetic} --- there exists $M$ such that any $G_k$-periodic point of $E_{\bar\rho, c}$ is fixed by $\on{Frob}^M$.  Consider the set of all $\on{Frob}^M$-fixed points; this is an analytic subset of $E_{\bar\rho, c}$.  By the Weierstrass preparation theorem, if it is infinite, it is in fact a \emph{positive-dimensional} rigid space. Hence it contains a $\mathbb{C}_\ell$-point not defined over $\overline{\mathbb{Q}_\ell}$.  But such a point would be the pseudorepresentation associated to a semisimple arithmetic representation (by \cite[Theorem 1]{taylor}) over $\mathbb{C}_\ell$, not defined over $\overline{\mathbb{Q}_\ell}$, contradicting Theorem \ref{rigidity-theorem}.

\emph{Step 2.} We now reduce to the case of finite fields.  Let $X'$ be the finite \'etale cover of $X_{\bar k}$ defined by $\on{ker}(\bar\rho)\subset \pi_1(X_{\bar k}, \bar x)$.  After replacing $k$ with a finite extension, we may assume that $X'$ and the finite \'etale map $X'\to X_{\bar k}$ are in fact defined over $k$.  Let $\bar x'$ be a geometric point of $X'$ lying over $\bar x$.  It suffices to prove the theorem with $X$ replaced by $X'$ and $\bar \rho$ replaced with the trivial representation. Indeed, any semisimple representation $\rho$ of $\pi_1^{\text{\'et}}(X_{\bar k}, \bar x)$ is a subquotient of $\on{Ind}_{\pi_1^{\text{\'et}}(X'_{\bar k}, \bar x')}^{\pi_1^{\text{\'et}}(X_{\bar k}, \bar x)}(\rho|_{\pi_1^{\text{\'et}}(X'_{\bar k}, \bar x')})$. There are finitely many such subquotients.  So we rename $X'$ as $X$ and assume $\bar\rho$ is trivial.

Now let $\overline{X}$ be the unique smooth geometrically connected curve containing $X$, and let $D=\overline{X}\setminus X$; after extending $k$ we may assume $D=\{x_1, \cdots, x_i\}$ where the $x_i$ are rational points of $\overline{X}$; after a further extension, we may assume the geometric point $\bar x$ of $X$ comes from a $k$-rational point $x$ of $X$.  There exists an algebra $R\subset k$, finitely generated over $\mathbb{Z}$, and a smooth curve $\overline{\mathscr{X}}/R$ with disjoint $R$-points $\xi_0, \xi_1, \cdots, \xi_i$ so that $(\overline{\mathscr{X}}, \xi_0, \xi_1, \cdots, \xi_i)_{k}\simeq (X,  x, x_1, \cdots, x_i)$.  For any closed point $z\in \on{Spec}(R)$ with $\on{char}(\kappa(z))\not=\ell$, the specialization map 
\begin{equation}\label{specialization-iso}
\pi_1^{\ell}(X_{\bar k}, \bar x)\to \pi_1^\ell((\overline{\mathscr{X}}\setminus\{\xi_1, \cdots, \xi_i\})_{\overline{\kappa(z)}}, (\xi_0)_{\overline{\kappa(z)}})
\end{equation}
 is an isomorphism, where $\pi_1^\ell$ denotes the pro-$\ell$ completion of $\pi_1^{\text{\'et}}$ (see \cite[Theorem A.10]{lieblich-olsson}, for example). Choose such a $z$,  let $\overline{\kappa(z)}$ be an algebraic closure of $\kappa(z)$, $\bar z\in \on{Spec}(S)(\overline{\kappa(z)})$ the asssociated geometric point of $\on{Spec}(S)$, and let $F\in \pi_1^{\text{\'et}}(\on{Spec}(S), \bar z)$ be the Frobenius element.

Now we claim that any arithmetic representation $\rho$ of $\pi_1^{\text{\'et}}(X_{\bar k}, \bar x)$ which is residually trivial (and hence factors through $\pi_1^{\ell}(X_{\bar k}, \bar x)$) gives an arithmetic representation of  $\tilde\rho: \pi_1^{\text{\'et}}((\overline{\mathscr{X}}\setminus\{\xi_1, \cdots, \xi_i\})_{\overline{\kappa(z)}}, (\xi_0)_{\overline{\kappa(z)}})$, via the isomorphism \ref{specialization-iso}. Indeed, it suffices to show that there exists $M$ such that $\tilde\rho^{F^M}\simeq \tilde\rho$.  But the action of $G_k$ on $\rho$ factors through its action on $\pi_1^\ell(X_{\bar k}, \bar x)$ and hence through $\pi_1^{\text{\'et}}(\on{Spec}(S))$; by arithmeticity, there is a finite index subgroup of $\pi_1^{\text{\'et}}(\on{Spec}(S))$ fixing $\rho$.  Thus $F^M$ fixes $\rho$ and hence $\tilde\rho$.

Thus we may replace $k$ with $\kappa(z)$ and $X$ with $\overline{\mathscr{X}}\setminus\{\xi_1, \cdots, \xi_i\}$; as $\kappa(z)$ is finite, we have reduced to the case of finite fields, which was proven in Step 1.
\end{proof}
We now deduce Corollary \ref{main-corollary}.
\begin{proof}[Proof of Corollary \ref{main-corollary}]
We first prove the statements about semisimple arithmetic representations.

Let $L$ be a finite extension of $\mathbb{Q}_\ell$, as in the statement, with valuation ring $\mathscr{O}_L$, maximal ideal $\mathfrak{m}_L$, and residue field $\mathbb{F}_{\ell^r}$. There are finitely many continuous representations $$\pi_1^{\text{\'et}}(X_{\bar k}, \bar x)\to GL_n(\mathbb{F}_{\ell^r}),$$ (resp.~$\pi_1^{\text{tame}}(X_{\bar k}, \bar x)\to GL_n(\mathbb{F}_{\ell^r})$ in characteristic $p>0$) as $\pi_1^{\text{\'et}}(X_{\bar k}, \bar x)$ (resp.~~$\pi_1^{\text{tame}}(X_{\bar k}, \bar x)$) is topologically finitely-generated.  Given this, (2) and (3) follow from (1).  So we fix $\bar\rho$ as in the statement; we wish to show that there are finitely many semisimple arithmetic $GL_n(L)$-valued representations $\rho$ with $\on{Tr}(\rho)\equiv \on{Tr}(\bar\rho)\bmod \mathfrak{m}_L$.

Let $c\in \mathbb{R}$ be such that $0<c<v_\ell(x)$ for any $x\in \mathfrak{m}_L$; such a $c$ exists because $L$ is discretely valued.  Now any arithmetic representation $\rho$ admits a $\pi_1^{\text{\'et}}(X_{\bar k}, \bar x)$-stable $\mathscr{O}_L$-lattice $M$; let $\bar\rho_M=M/\mathfrak{m}_LM$.  By our choice of $c$, we have $$\on{Tr}(\rho)\equiv \on{Tr}(\bar \rho)\bmod \ell^c.$$ Hence there are only finitely many possibilities for $\rho$, by Theorem \ref{main-theorem}.

We now deduce the required statements for representations which arise from geometry.  But such representations are arithmetic by Proposition \ref{comes-from-geometry}, and semisimple by \cite[Corollaire 3.4.13]{weilii}.
\end{proof}
\section{An analogue of the Frey-Mazur conjecture}
We now begin preparations for the proof of Theorem \ref{ball-bound}.
\subsection{Weight filtrations on deformation rings}
Let $C$ be a smooth, affine, geometrically connected curve over a finitely generated field $k$ of characteristic $0$, and let $c$ be a rational point of $C$; choosing an algebraic closure of $k$, $c$ gives rise to a geometric point $\bar c$.  Let $\overline{C}$ be the unique smooth, proper, geometrically connected $k$-curve containing $C$ as an open subscheme.  Let $$\rho: \pi_1^{\text{\'et}}(C_{\bar k}, \bar c)\to GL_n(\overline{\mathbb{Q}_\ell})$$ be an arithmetic representation of $\pi_1^{\text{\'et}}(C_{\bar k}, \bar c)$.  The Leray spectral sequence for the inclusion $$j: C\hookrightarrow \overline{C}$$ gives an exact sequence 
\begin{equation}\label{leray-sequence}
0\to H^1(\overline{C}_{\bar k, \text{\'et}}, j_*\rho) \to H^1(C_{\bar k,\text{\'et}}, \rho)\to H^0(\overline{C}_{\bar k, \text{\'et}}, R^1j_*\rho)\to H^2(\overline{C}_{\bar k, \text{\'et}}, j_*\rho)
\end{equation}
where we here identify $\rho$ with the associated lisse $\ell$-adic sheaf.  
\begin{defn}\label{h1-weight-filtration}
The \emph{weight filtration} on $H^1(C_{\bar k,\text{\'et}}, \rho)$ is defined by $$W^iH^1(C_{\bar k,\text{\'et}}, \rho)=0 \text{ for } i\leq 0$$ $$W^1H^1(C_{\bar k,\text{\'et}}, \rho)=H^1(\overline{C}_{\bar k, \text{\'et}}, j_*\rho)$$ $$W^iH^1(C_{\bar k,\text{\'et}}, \rho)=H^1(C_{\bar k,\text{\'et}}, \rho) \text{ for } i\geq 2.$$
\end{defn}
By \cite[Th\'{e}or\`{e}me 2]{weilii}, this agrees with the usual, geometrically-defined weight filtration on $H^1(C_{\bar k,\text{\'et}}, \rho)$ if $\rho$ arises from geometry and is pure of weight zero.

Now suppose $\rho$ is irreducible, and let $S_\rho$ be its deformation ring.  Note that for two given basepoints $\bar c, \bar c'$ of $C$, the associate deformation rings are canonically isomorphic, so we do not include $\bar c$ in the notation.  Let $$\rho_{\text{univ}}: \pi_1^{\text{\'et}}(C_{\bar k}, \bar c)\to GL_n(S_\rho)$$ be the universal deformation of $\rho$.

Because $C$ is an affine curve, $$S_{\rho}\simeq \overline{\mathbb{Q}_\ell}[[x_1, \cdots, x_N]]$$ non-canonically, where $N=\dim H^1(C_{\bar k, \text{\'et}}, \rho\otimes \rho^\vee)$.  Let $\mathfrak{m}_\rho$ be the maximal ideal of $S_\rho$.  We have $$\mathfrak{m}_\rho/\mathfrak{m}_\rho^2\simeq H^1(C_{\bar k, \text{\'et}}, \rho\otimes \rho^\vee)^\vee$$ canonically; let $W^i(\mathfrak{m}_\rho/\mathfrak{m}_\rho^2)$ be the dual filtration to the filtration given in Definition \ref{h1-weight-filtration}.
\begin{defn}[Weight filtration on $S_\rho$]
Let $$W^iS_\rho=S_\rho \text{ for } i\geq 0,$$ $$W^{-1}S_\rho=\mathfrak{m}_\rho$$ $$W^{-2}S_\rho=\mathfrak{m}_\rho^2+W^{-2}(\mathfrak{m}_\rho/\mathfrak{m}_\rho^2)$$ and $$W^{-m}S_\rho=\sum_{i+j=m} (W^{-i}S_\rho)\cdot (W^{-j}S_\rho) \text{ for } m>2.$$
\end{defn}
\begin{remark}
If $C$ is proper, $W^{-i}=\mathfrak{m}_{\rho}^i$ for $i\geq 0$. In general, it is immediate from the definition that $$W^{-2i}\subset \mathfrak{m}_\rho^i\subset W^{-i}$$ for $i\geq 0$.
\end{remark}
Replace $k$ with a finite extension so that $G_k$ acts on $S_\rho$. The goal of this section is to construct, for $\alpha\in \mathbb{Z}_\ell^\times$ sufficiently close to $1$, an element $\sigma_\alpha\in G_k$ such that $\sigma_\alpha$ acts on $\on{gr}^{-i}_W S_\rho$ via the scalar $\alpha^i$.  We first prove the analogous statement for $H^1(C_{\bar k,\text{\'et}}, \rho\otimes \rho^\vee)$, where $\rho$ is irreducible and arises from geometry.

So we assume $\rho$ is irreducible and arises from geometry. Note that, as $\rho$ is irreducible, there exists (after replacing $k$ with a finite extension) a representation $$\tilde\rho: \pi_1^{\text{\'et}}(C_{k}, \bar c)\to GL_n(\overline{\mathbb{Q}_\ell})$$ such that $\tilde\rho|_{\pi_1^{\text{\'et}}(C_{\bar k}, \bar c)}\simeq \rho,$ and such that the representation $\rho_G$ of $G_{k}$ given by restricting $\tilde\rho$ to the decomposition group at $c$ is de Rham at primes above $\ell$, and unramified almost everywhere.  Moreover, the irreducibility of $\rho$ implies that $\rho_G\otimes \rho_G^\vee$ is independent of the choice of $\rho'$; in particular, $G_k$ acts canonically on $H^1(C_{\bar k}, \rho\otimes \rho^\vee)$, and as $\rho\otimes\rho^\vee$ is pure of weight zero, Weil II \cite{weilii} implies $H^1(C_{\bar k}, \rho\otimes \rho^\vee)$ is mixed with weights in $\{1,2\}$, with the weight filtration given as in Definition \ref{h1-weight-filtration}.
\begin{lemma}\label{h1-quasiscalar}
Let $$\rho: \pi_1^{\text{\'et}}(C_{\bar k}, \bar c)\to GL_n(\overline{\mathbb{Q}_\ell})$$ be an irreducible representation which arises from geometry.  Then for $\alpha\in \mathbb{Z}_\ell^\times$ sufficiently close to $1$, there exists $\sigma_\alpha\in G_k$ such that $\sigma_\alpha$ acts on $\on{gr}^1_W H^1(C_{\bar k}, \rho\otimes \rho^\vee)$ via $\alpha\cdot \on{Id}$, and on $$(\on{gr}^2_WH^1(C_{\bar k}, \rho \otimes \rho^\vee))\oplus H^0(\overline{C}_{\bar k, \text{\'et}}, R^1j_*(\rho\otimes \rho^\vee))$$ via $\alpha^2\cdot \on{Id}.$
\end{lemma}
\begin{proof}
This is similar to \cite[Lemma 2.10]{litt}, \cite[Lemme 12]{hindry}, or \cite[Theorem 3]{bogomolov-points}, with some mild additional complications arising from the fact that we do not assume $\rho$ arises from the monodromy action on the Tate module of an Abelian $C$-scheme. For simplicity of notation, we set $$V=H^1(C_{\bar k}, \rho \otimes \rho^\vee)\oplus H^0(\overline{C}_{\bar k, \text{\'et}}, R^1j_*\rho),$$ and we let $$\gamma: G_k\to GL(V)$$ be the Galois representation we are studying. We assume $k$ is a number field; the general case follows by the argument of \cite[Letter to Ribet of 1/1/1981, \S 1]{serre-oeuvres}.  The weight filtration on $V$ is inherited from that of $H^1(C_{\bar k}, \rho \otimes \rho^\vee)$; $H^0(\overline{C}_{\bar k, \text{\'et}}, R^1j_*\rho)$ is pure of weight $2$.

\emph{Step 1.} We first show that for any $\alpha\in \mathbb{Q}_\ell^\times$, the Zariski-closure of $\on{im}(\gamma)$ contains elements $\sigma_\alpha$ acting as required.  Let $F\in G_k$ be a Frobenius element acting on $\on{gr}^i_W(V)$ with weight $i$; let $Z\subset \overline{\on{im}(\gamma)}$ be the identity component of Zariski-closure of $\{F^n\}_{n\in \mathbb{Z}}$.  $Z$ is a commutative, connected, algebraic group over a field of characteristic zero, and hence $Z\simeq T\times U$ canonically, where $T$ is a torus and $U$ is unipotent.  After replacing $F$ with a power, we may assume it lies in $Z$.  In particular $F$ admits a unique decomposition $F=F_sF_u$, where $F_s\in T$ is semisimple and $F_u\in U$; $F_s$ and $F$ have the same eigenvalues.  Choose a basis of eigenvectors $\{e_i\}_{i=1, \cdots ,\dim(V)}$ for $F_s$, with eigenvalues $\{\lambda_i\}$, and let $D\subset GL(V)$ be the diagonal torus for this basis.  Let $I_1\subset \{1, \cdots, \dim(V)\}$ be the set of indices $i$ such that $\lambda_i$ has weight $1$, and $I_2$ the set of indices $i$ such that $\lambda_i$ has weight $2$.  The inclusion $T\hookrightarrow D$ induces a surjection on cocharacter lattices $X(D)\twoheadrightarrow X(T)$ with kernel $K$; $T$ is precisely the subtorus of $D$ cut out by the characters in $K$. If we identify $D$ with $\mathbb{Z}^{\dim V}$ via the choice of basis $\{e_i\}$, $K$ consists of the vectors $\underline{a}=(a_1, \cdots, a_{\dim V})$ such that $$\prod_i \lambda_i^{a_i}=1.$$ But if $\underline{a}\in K$, $$\prod_i |\lambda_i|^{a_i}=1,$$ where $|\cdot|$ denotes any Archimedean absolute value on $\overline{\mathbb{Q}}$, by the multiplicativity of absolute value.  In other words, $$\prod_{i\in I_1} q^{a_i/2} \cdot \prod_{j\in I_2} q^{a_j}=1,$$ where $q$ is the size of the residue field of the prime corresponding to $F$, and hence $$\sum_{i\in I_1} a_i+2\sum_{j\in I_2} a_j=0.$$

Hence if $V=\on{gr}^1_W(V)\oplus \on{gr}^2_W(V)$ is the unique $F_s$-equivariant splitting, we have that $$\alpha\cdot\on{Id}_{\on{gr}^1_W(V)}\bigoplus \alpha^2\cdot\on{Id}_{\on{gr}^2_W(V)}\in T,$$ as desired.

\emph{Step 2.}
Let $\tilde\rho: \pi_1^{\text{\'et}}(C_{k}, \bar c)\to GL_n(\overline{\mathbb{Q}_\ell})$ be an extension of $\rho$ to the arithmetic fundamental group, as in the discussion before the statement of the lemma; then $\tilde\rho$ is defined over $L$ for some $L/\mathbb{Q}_\ell$ finite. We abuse notation and write $$\tilde\rho: \pi_1^{\text{\'et}}(C_{k}, \bar c)\to GL_n(L)$$ for some choice of descent of our original repesentation to a model over $L$.   Let $V_L$ be the descent of $V$ to $L$ obtained from $\tilde\rho$, and let $$\tilde \gamma: G_k\to GL(V_L)$$ be the associated Galois representation.

Let $$\gamma'=W_{L/\mathbb{Q}_\ell}\tilde\gamma$$ be the representation obtained by Weil restriction of $\tilde\gamma$ to $\mathbb{Q}_\ell$. 

We wish to show that the image of $\gamma'$ is open in the $\mathbb{Z}_\ell$-points of its Zariski-closure; by Step 1, this suffices. By \cite[Th\'{e}or\`{e}me 1]{bogomolov}, it is enough to check that $\gamma'$ is Hodge-Tate at primes above $\ell$.  But for any prime $v$ above $\ell$, the lisse sheaf associated to $\tilde\rho\otimes \tilde\rho^\vee$ is a de Rham local system on $C_{k_v}$ in the sense of \cite{liu-zhu}, and hence its cohomology group $$H^1(C_{\bar k}, \tilde\rho \otimes \tilde\rho^\vee)$$ is de Rham by \cite[Theorem 1.1]{lan-liu-zhu}, for example. 

Moreover, we claim that $H^0(\overline{C}_{\bar k, \text{\'et}}, R^1j_*\rho)$ is de Rham. This is a local computation at each point of $\overline{C}\setminus C$; after replacing the local ring at each point by a ramified extension, we may assume $\rho$ comes from the cohomology of a semistable $\overline{C}$-scheme and conlcude by e.g.~the weight spectral sequence for this semistable scheme.  (Alternately, this can be deduced from \cite[Lemma 5.57]{lan-liu-zhu}.) Hence it is Hodge-Tate, which completes the proof. 
\end{proof}
\begin{remark}\label{fontaine-mazur-remark}
This lemma is the only place in the proof of Theorem \ref{ball-bound} in which the geomtricity of $\rho$ is used.  In fact, the weaker condition that $\rho$ extends to a local system on $C$ which is \emph{geometric in the sense of Fontaine-Mazur} suffices --- see e.g. the conjecture on page 2 of \cite{liu-zhu} for a discussion of this notion.  In this case one may deduce the desired $p$-adic Hodge-theoretic properties for $H^1(C_{\bar k}, \rho\otimes \rho^\vee), H^0(\overline{C}_{\bar k}, R^1j_*(\rho\otimes \rho^\vee))$ for such $\rho$ from \cite[Theorem 1.1 and Lemma 5.57]{lan-liu-zhu}, for example.
\end{remark}

The set of $\alpha$ for which the desired $\sigma_\alpha$ exist is an important invariant of the representation $\rho$, which we record in the following definition.

\begin{defn}[Index of homothety]\label{index-of-homothety}
Let $\rho$ be as in Lemma \ref{h1-quasiscalar}.  Let $Z\subset \mathbb{Z}_\ell^{\times}$ be the set of $\alpha$ for which there exists $\sigma_\alpha$ satisfying the conclusions of Lemma \ref{h1-quasiscalar}.  Then the \emph{index of homothety} of $\rho$, denoted $c(\rho)$, is the index of $Z$ in $\mathbb{Z}_\ell^\times$. By Lemma \ref{h1-quasiscalar}, $Z$ is open in $\mathbb{Z}_\ell^\times$, so $c(\rho)$ is finite.
\end{defn}

We now give the analogous statement for $S_\rho$.
\begin{theorem}\label{pi1-quasiscalar}
Let $$\rho: \pi_1^{\text{\'et}}(C_{\bar k}, \bar c)\to GL_n(\overline{\mathbb{Q}_\ell})$$ be an irreducible representation which arises from geometry.  Then for $\alpha\in \mathbb{Z}_\ell^\times$ sufficiently close to $1$, there exists $\sigma_\alpha\in G_k$ such that $\sigma_\alpha$ acts on $\on{gr}^i_W S_\rho$ via $\alpha^i\cdot \on{Id}$.
\end{theorem}
Before giving the proof, we need to analyze the contribution of the inertia subgroups of $\pi_1^{\text{\'et}}(C_{\bar k}, \bar c)$ to the geometry of $S_\rho$.

Recall that $\overline{C}$ is the smooth curve compactifying $C$. Let $D=\overline{C}\setminus C$, and pass to a finite extension of $k$ so that $D=\{x_1, \cdots, x_n\}$ for rational points $x_i\in \overline{C}(k)$.  For $i=1, \cdots, n$, let $R_i$ be the local ring of $\overline{C}$ at $x_i$; by the Cohen structure theorem, each $R_i\simeq k[[t]]$. Let $\overline{R_i}$ be the completion of $R_i\otimes \overline{k}$ at its maximal ideal, and let $z_i= \on{Spec}(\on{Frac}(\overline{R_i}))$.  Let $$\eta_i:z_i \to C_{\bar k}$$ be the natural inclusion.  A local computation shows that $$H^0(\overline{C}_{\bar k}, R^1j_*(\rho\otimes \rho^\vee))\simeq \bigoplus_i H^1(z_i, \eta_i^*(\rho \otimes \rho^\vee)),$$ and that under this identification, the map 
\begin{equation}\label{tangent-map}
H^1(C_{\bar k}, \rho\otimes \rho^\vee)\to H^0(\overline{C}_{\bar k}, R^1j_*(\rho\otimes \rho^\vee)\simeq \bigoplus_i H^1(z_i, \eta_i^*(\rho \otimes \rho^\vee))
\end{equation}
 is given by $\oplus_i \eta_i^*$.

Let $\overline{K_i}$ be an algebraic closure of $K_i:=\on{Frac}(\overline{R_i})$, and let $\bar z_i$ be the geometric point of $C$ associated to $z_i$ by this choice.  The inclusion $\eta_i: z_i\to C_{\bar k}$ induces a galois-equivariant inclusion $\gamma_i: \widehat{\mathbb{Z}}(1)\simeq\on{Gal}(\overline{K_i}/K_i)\hookrightarrow \pi_1^{\text{\'et}}(C_{\bar k}, \bar z_i)$. Let $\rho_i=\rho\circ \gamma_i$ be the restriction of $\rho$ to $\on{Gal}(\overline{K_i}/K_i)$.  (Here we view $\rho$ as a representation of  $\pi_1^{\text{\'et}}(C_{\bar k}, \bar z_i)$ rather than $\pi_1^{\text{\'et}}(C_{\bar k}, \bar c)$ via a change-of-basepoint isomorphism, well-defined up to conjugacy. Recall from before that for any two choices of basepoint, the resulting deformation rings $S_\rho$ are canonically isomorphic, so this indeterminacy will not affect our later constructions.)

Let $S_i$ be a hull for the deformation functor $D_{\rho_i}$ for $\rho_i$ (that is, the functor which associates to an Artinian $\overline{\mathbb{Q}_\ell}$-algebra $A$ the set of isomorphism classes of representations $\on{Gal}(\overline{K_i}/K_i)\to GL_n(A)$ lifting $\rho_i$).  As $K_i$ has cohomological dimension $1$, $S_i$ is a smooth complete local Noetherian $\overline{\mathbb{Q}_\ell}$-algebra, with $$\mathfrak{m}_i/\mathfrak{m}_i^2\simeq H^1(z_i, \eta_i^*(\rho \otimes \rho^\vee))^\vee$$ canonically, where $\mathfrak{m}_i$ is the maximal ideal of $S_i$.

The map $\gamma_i$ gives a $S_\rho$-point of $D_{\rho_i}$; we may choose a lift to $S_i$, giving a map $\gamma_i^*: S_i\to S_\rho$ for each $i$. We must deal with some subtle issues coming from the non-canonicity of this choice.
\begin{lemma}\label{horrible-lemma}
Suppose $|D|\geq 2$. Then 
\begin{enumerate}
\item the natural transformation $D_{\rho}\to D_{\rho_i}$ is an epimorphism,
\item any lift of this map to a map $\gamma_i^*: S_i\to S_\rho$ is injective, 
\item its image is stable under the action of $G_k$, and
\item the induced $G_k$-action on $S_i$ lifts its action on $D_{\rho_i}$.
\end{enumerate}
\end{lemma}
\begin{proof}
We first prove (1).  It is enough to show that for each Artin $\overline{\mathbb{Q}_\ell}$-algebra $A$, the map $D_\rho\to D_{\rho_i}$ is surjective.  Let $$\tilde\rho_i: \on{Gal}(\overline{K_i}/K_i)\to GL_n(A)$$ be an $A$-point of $D_{\rho_i}$.  We wish to lift it to an $A$-point of $D_\rho$.

Recall that if $C$ has genus $g$, then $$\pi_1^{\text{\'et}}(C_{\bar k}, \bar z_i)\simeq \langle a_1,b_1, \cdots, a_g, b_g,\lambda_1, \cdots, \lambda_{|D|} \big| \prod_{i=1}^g [a_i, b_i] \cdot \prod_{j=1}^{|D|} \lambda_j\rangle^{\widehat{}},$$ 
where without loss of generality, $\on{Gal}(\overline{K_i}/K_i)\subset \pi_1^{\text{\'et}}(C_{\bar k}, \bar z_i)$ is topologically generated by $\lambda_1$.  We wish to find a representation $\tilde\rho$ of $\pi_1^{\text{\'et}}(C_{\bar k}, \bar z_i)$ into $GL_n(A)$, lifting $\rho$, where the value of $\lambda_1$ is specified, say $\tilde\rho(\lambda_1)=M$.  But we may choose arbitrary lifts of $\rho(a_1), \rho(b_1), \cdots, \rho(a_g), \rho(b_g)$, set $\tilde\rho(\lambda_1)=M,$ choose $\tilde(\rho(\lambda_i))$ arbitrary lifts of $\rho(\lambda_i)$ for $i=2, \cdots, |D|-1$, and set $$\tilde\rho(\lambda_{|D|})=\left(\prod_{i=1}^g [\tilde\rho(a_i), \tilde\rho(b_i)] \cdot \prod_{j=1}^{|D|-1} \tilde\rho(\lambda_j)\right)^{-1}.$$

We now prove (2).  $S_i, S_\rho$ are both power series rings over $\overline{\mathbb{Q}_\ell}$, since $\on{Gal}(\overline{K_i}/K_i), \pi_1^{\text{\'et}}(C_{\bar k}, \bar z_i)$ have cohomological dimension $1$.  So it is enough to show that the induced map $\mathfrak{m}_i/\mathfrak{m}_i^2\to \mathfrak{m}_\rho/\mathfrak{m}_\rho^2$ is injective.  But this follows by applying (1) to the map $$D_\rho(\overline{\mathbb{Q}_\ell}[\epsilon]/\epsilon^2)\to D_{\rho_i}(\overline{\mathbb{Q}_\ell}[\epsilon]/\epsilon^2).$$

We now prove (3). The image of $\gamma_i^*$ is generated by the matrix entries of $\rho_{\text{univ}}(\on{im}(\gamma_i))$.  But $G_k$  preserves $\on{im}(\gamma_i)$ by our choice of basepoint. Hence the image is stable under the $G_k$-action, as desired. Hence we have a $G_k$-action on $S_i$ so that $\gamma_i^*$ is $G_k$-equivariant, by the injectivity of $\gamma_i^*$.

Finally, we prove (4). For $\sigma\in G_k$, we wish to show that the diagram
$$\xymatrix{
S_i\ar[r]^{\sigma} \ar[d] & S_i\ar[d]\\
D_{\rho_i} \ar[r]^{\sigma} & D_{\rho_i}
}$$
commutes, where $\sigma$ acts as described in the previous paragraph.  But this follows from the fact that the diagram 
$$\xymatrix{
D_\rho \ar[r]^{\sigma} \ar[d] & D_\rho\ar[d]\\
D_{\rho_i} \ar[r]^{\sigma} & D_{\rho_i}
}$$
commutes, and the fact that $$D_\rho\overset{\gamma_i^*}{\to}  \on{Hom}(S_i, -)$$ is an epimorphism, by (2). 
\end{proof}
\begin{proof}[Proof of Theorem \ref{pi1-quasiscalar}]
It suffices to prove the theorem after deleting several closed points of $C$, so we may assume that $|D_{\bar k}|\geq 2$, where $D=\overline{C}\setminus C$.

From Lemma \ref{h1-quasiscalar}, we know that for $\alpha\in \mathbb{Z}_\ell^{\times}$ sufficiently close to $1$, there exists $\sigma_\alpha\in G_k$ acting on $$\on{gr}^i_W \mathfrak{m}_\rho/\mathfrak{m}_\rho^2=\on{gr}^i_W H^1(C_{\bar k}, \rho\otimes \rho^\vee)^\vee$$ via $\alpha^i\cdot\on{Id}$.  We claim that such a $\sigma_\alpha$ also acts on $\on{gr}^i_W S_\rho$ in the desired manner.

\emph{Step 1.}  We first claim that it suffices to show that the exact sequence 
\begin{equation} \label{m-sequence}
0\to \mathfrak{m}_\rho^2/\mathfrak{m}_\rho^3\to \mathfrak{m}_\rho/\mathfrak{m}_\rho^3\to \mathfrak{m}_\rho/\mathfrak{m}_\rho^2\to 0
\end{equation}
splits $\sigma_\alpha$-equivariantly.  Indeed, multiplication gives an isomorphism $$\on{Sym}^i(\mathfrak{m}_\rho/\mathfrak{m}_\rho^2)\to \mathfrak{m}_\rho^i/\mathfrak{m}_\rho^{i+1},$$ so the eigenvalues of the $\sigma_\alpha$ action on $\mathfrak{m}_\rho^i/\mathfrak{m}_\rho^{i+1}$ are contained in $\{\alpha^i, \cdots, \alpha^{2i}\}$.  Thus (by the completeness of $S_\rho$), $\mathfrak{m}_\rho\to \mathfrak{m}_\rho/\mathfrak{m}_\rho^2$ splits $\sigma_\alpha$-equivariantly if and only if sequence (\ref{m-sequence}) does, because neither $\alpha$ nor $\alpha^2$ appears as a generalized eigenvalue for the $\sigma_\alpha$-action on $\mathfrak{m}_\rho^3$.  Such a splitting induces a $\sigma_\alpha$-equivariant isomorphism $$\on{Sym}^*(\mathfrak{m}_\rho/\mathfrak{m}_\rho^2)\overset{\sim}{\to} S_\rho,$$ which respects the weight filtrations, where the weight filtration on $\on{Sym}^*(\mathfrak{m}_\rho/\mathfrak{m}_\rho^2)$ is induced from the filtration on $\mathfrak{m}_\rho/\mathfrak{m}_\rho^2$, by the multiplicativity of the weight filtration. The element $\sigma_\alpha\in G_k$ clearly acts on $\on{Sym}^*(\mathfrak{m}_\rho/\mathfrak{m}_\rho^2)$ as desired (again by multiplicativity), so we are done.

\emph{Step 2.} We now show the sequence (\ref{m-sequence}) does indeed split $\sigma_\alpha$-equivariantly.  Write $$\mathfrak{m}_\rho/\mathfrak{m}_\rho^2=V_1\oplus V_2,$$ where $V_i$ is the $\alpha^i$-eigenspace of $\sigma_\alpha$.  The $\sigma_\alpha$-action on $\mathfrak{m}_\rho^2/\mathfrak{m}_\rho^3$ has eigenvalues in $\{\alpha^2, \alpha^3, \alpha^4\}$, so the projection $$\mathfrak{m}_\rho/\mathfrak{m}_\rho^3\to \mathfrak{m}_\rho/\mathfrak{m}_\rho^2\to V_1$$ splits $\sigma_\alpha$-equivariantly. Thus it suffices to show that the projection $$\mathfrak{m}_\rho/\mathfrak{m}_\rho^3\to \mathfrak{m}_\rho/\mathfrak{m}_\rho^2\to V_2$$ splits $\sigma_\alpha$-equivariantly.

We first give an explicit description of the subspace $V_2\subset \mathfrak{m}_\rho/\mathfrak{m}_\rho^2$. Recall that $\mathfrak{m}_\rho/\mathfrak{m}_\rho^2\simeq H^1(C_{\bar k}, \rho\otimes \rho^\vee)^\vee$ canonically.  As before, let $\overline{C}$ be the unique smooth, geometrically connected, proper curve containing $C$, and $j: C\hookrightarrow \overline{C}$ the natural inclusion.  By sequence (\ref{leray-sequence}), $V_2$ is the image of the dual of the natural map $$H^1(C_{\bar k}, \rho\otimes \rho^\vee)\to H^0(\overline{C}_{\bar k}, R^1j_*(\rho\otimes \rho^\vee))$$ (the second map in sequence (\ref{leray-sequence})).  Recall that this map is described in line (\ref{tangent-map}), i.e. it is given by $\oplus_i \eta_i^*,$ where we identify $$H^0(\overline{C}_{\bar k}, R^1j_*(\rho\otimes \rho^\vee)\simeq \bigoplus_i H^1(z_i, \eta_i^*(\rho \otimes \rho^\vee))$$ and $$\eta_i^*: H^1({C}_{\bar k}, \rho\otimes \rho^\vee)\to H^1(z_i, \eta_i^*(\rho\otimes \rho^\vee))$$ is induced by the inclusion $$\eta_i: \overline{z_i}\to C_{\bar k}.$$

Now let $S_i, \mathfrak{m}_i$ be as in Lemma \ref{horrible-lemma}. The map $\mathfrak{m}_i/\mathfrak{m}_i^2\to \mathfrak{m}_\rho/\mathfrak{m}_\rho^2$ induced by $\gamma_i$ is dual to $\eta_i^*$, so $$\bigoplus_i \eta_i^{*\vee}: \bigoplus_i \mathfrak{m}_i/\mathfrak{m_i}^2\to V_2$$ is surjective.  Thus it suffices to lift this map to a $\sigma_\alpha$-equivariant map $$\bigoplus_i \mathfrak{m}_i/\mathfrak{m_i}^2\to S_\rho.$$

By Lemma \ref{horrible-lemma}, we have a Galois-equivariant map $$\gamma_i^*: S_i\to S_\rho,$$ equivariantly lifting the map $D_\rho\to D_{\rho_i}$.  The element $\sigma_\alpha\in G_k$ acts on $\mathfrak{m}_i^r/\mathfrak{m}_i^{r+1}$ via $\alpha^{2i}\cdot \on{Id}$, so the decomposition of $S_i$ into $\sigma_\alpha$-eigenspaces gives an isomorphism $$S_i\simeq \prod_{r\geq 0} \mathfrak{m}_i^r/\mathfrak{m}_i^{r+1}.$$  Let $p_i: \mathfrak{m}_i/\mathfrak{m}_i^2\to S_\rho$  the map induced by this decomposition.

Then the map $$\bigoplus_i p_i: \bigoplus_i \mathfrak{m}_i/\mathfrak{m}_i^2\to \mathfrak{m}_\rho\to S_\rho$$ is the desired lift of $\bigoplus_i \eta_i^{*\vee}$, completing the proof.
\end{proof}
\begin{remark}
We could have split the short exact sequence (\ref{m-sequence}) directly, giving a slightly more explicit proof of Theorem \ref{pi1-quasiscalar}.  If we let $\iota_i\in \on{Gal}(\overline{K_i}/K)$ be a generator, it is not hard to see that $\sigma_\alpha(\iota_i)=\iota_i^{\alpha^2}$. We have that $\rho(\iota_i)$ is quasi-unipotent by the geometricity of $\rho$; let $N$ be an integer such that $\rho(\iota_i)^N$ is unipotent, so that $\log(\rho_{\text{univ}}(\iota_i^N))$ converges.  Then $\sigma_\alpha$ acts on the matrix entries of $\log(\rho_{\text{univ}}(\iota_i^N))$ via multiplication by $\alpha^2$, and one may check directly that as one varies over all $i$, the span of these matrix entries surjects onto $V_2\subset \mathfrak{m}_\rho/\mathfrak{m}_\rho^2$, providing the desired $\sigma_\alpha$-equivariant splitting. The proof above, though morally the same, is slightly more efficient, as Lemma \ref{horrible-lemma} saves us from checking that an unreasonable number of diagrams commute.
\end{remark}
\subsection{The integral analysis}
In the previous section, we constructed natural elements $\sigma_\alpha\in G_k$ whose action on the deformation ring of an irreducible arithmetic $\overline{\mathbb{Q}_\ell}$-representation we understand well. We now consider a geometrically irreducible representation $$\bar\rho:\pi_1^{\text{\'et}}(C_{\bar k}, \bar c)\to GL_n(\mathbb{F}_{\ell^r}),$$ and we use the results of the previous section to analyze the action of $\sigma_\alpha$ on $R_{\bar\rho}$, assuming $\bar\rho$ admits a lift which arises from geometry.
\begin{lemma}\label{denominators-lemma}
Let $V$ be a finite free $\overline{\mathbb{Z}_\ell}$-module, and $T: V\to V$ an endomorphism. Let $W\subset V$ be a $T$-stable submodule with $W, V/W$ free $\overline{\mathbb{Z}_\ell}$-modules.  Suppose that $T|_W=\alpha\cdot \on{Id}$, and that $T$ acts on $V/W$ via $\beta\cdot\on{Id}$, where $\alpha,\beta \in \overline{\mathbb{Z}_\ell}, \alpha\not=\beta$.  Let $v\in V\otimes \overline{\mathbb{Q}_\ell}$ be such that
\begin{enumerate}
\item $Tv=\beta v$, and
\item $v\in V+(W\otimes\overline{\mathbb{Q}_\ell})$.
\end{enumerate}
Then $(\alpha-\beta)\cdot v\in V.$
\end{lemma}
\begin{proof}
Choose a $\overline{\mathbb{Z}_\ell}$-basis $\{e_i\}$ of $V$ such that $e_1, \cdots, e_{\dim W}$ are a $\overline{\mathbb{Z}_\ell}$-basis of $W$.  Then $v=\sum a_i e_i$, where the $a_i\in \overline{\mathbb{Q}_\ell}$, and $a_i\in\overline{\mathbb{Z}_\ell}$ for $i>\dim W$.  We have $$Te_i=\begin{cases} \alpha  e_i & i\leq \dim W\\ \beta e_i +\sum_{j=1}^{\dim W} b_{ij} e_j & i>\dim W\end{cases}$$ where the $b_{ij}\in \overline{\mathbb{Z}_\ell}$.
Now 
\begin{align*}
\beta\cdot v &= Tv\\
&= \sum_{i=1}^{\dim W} \alpha a_i e_i +\sum_{i=\dim W+1}^{\dim V} \left(\beta a_i e_i+\sum_{j=1}^{\dim W} b_{ij} e_j\right)\\
&=\sum_{i=1}^{\dim W} (\alpha a_i+\sum_{j=1}^{\dim W} b_{ji})e_i+\sum_{i=\dim W+1}^{\dim V} \beta a_i e_i
\end{align*}
Equating coefficients for $e_i$, we have for each $i\leq \dim W$, $$\beta a_i=\alpha a_i+\sum_{j=1}^{\dim W} b_{ji}.$$  Hence $$a_i=\frac{\sum_{j=1}^{\dim W} b_{ji}}{\beta-\alpha},$$ whence the result follows.
\end{proof}
Now if $$\rho:\pi_1^{\text{\'et}}(C_{\bar k}, \bar c)\to GL_n(\overline{\mathbb{Q}_\ell})$$ is a lift of $\bar\rho$ (hence irreducible), the natural induced map $R_{\bar \rho}\to S_\rho$ exhibits $S_\rho$ as the completion of $R_{\bar\rho}{\otimes}\overline{\mathbb{Q}_\ell}$ at the maximal ideal corresponding to $\rho$.  The map $R_{\bar\rho}\to S_\rho$ gives $S_\rho$ a natural integral structure (namely the image of the induced map $R_{\bar\rho}\otimes \overline{\mathbb{Z}_\ell}\to S_\rho$). We now apply the computation in Lemma \ref{denominators-lemma} to estimate the denominators required to write down eigenvectors for the $\sigma_\alpha$-action on $S_\rho$, relative to this integral structure.
\begin{lemma}\label{eigenvector-lifting}
Let $$\rho: \pi_1(C_{\bar k}, \bar c)\to GL_n(\overline{\mathbb{Z}_\ell})$$ be a continuous representation lifting $\bar\rho\otimes \overline{\mathbb{F}_\ell}$, such that $\rho\otimes \overline{\mathbb{Q}_\ell}$  arises from geometry. Let $S^{\text{int}}_\rho \subset S_\rho$ be the image of the induced map $R_{\bar\rho}\otimes{\overline{\mathbb{Z}_\ell}}\to S_\rho$.

Let $\sigma_\alpha$ be as in Theorem \ref{pi1-quasiscalar}. Let $x\in S_\rho$ be such that 
\begin{enumerate}
\item $\sigma_\alpha\cdot x=\alpha^i x$, for $i=1$ or $i=2$, and
\item $x \in S^{\text{int}}_\rho+W^{-i-1}$.
\end{enumerate}
Then for any $r\geq i,$ $$\left(\prod_{j=i+1}^{r} (\alpha^i-\alpha^j)\right)\cdot x \in S^{\text{int}}_\rho+W^{-r-1}.$$ (Here if $r=i$ we take the product above to be the empty product, i.e.~$1$.)
\end{lemma}
\begin{proof}
This follows by induction on $r$ from Lemma \ref{denominators-lemma}. The case $r=i$ follows from hypothesis (2) of the lemma.  Now let $r>i$. Suppose the result holds for $r-1$; we now prove it for $r$.  Let $\bar x\in S_\rho^{\text{int}}$ be any element such that $$\bar x=\left(\prod_{j=i+1}^{r-1} (\alpha^i-\alpha^j)\right)\cdot x\bmod W^{-r}.$$ Set $$V=(S_\rho^{\text{int}}\cap W^{-r}+\overline{\mathbb{Z}_\ell}\cdot \bar x)/(S_\rho^{\text{int}}\cap W^{-r-1}),$$ $$W=(S_\rho^{\text{int}}\cap W^{-r})/(S_\rho^{\text{int}}\cap W^{-r-1}),$$ and $$v=\left(\prod_{j=i+1}^{r-1}(\alpha^i-\alpha^j)\right)\cdot x \bmod W^{-r-1}.$$  Then the hypotheses of Lemma \ref{denominators-lemma} are satisfied by the induction hypothesis, giving the proof. 
\end{proof}
\begin{lemma}\label{denominator-estimate}
Let $\alpha\in \mathbb{Z}_\ell^\times$ be an $\ell$-adic unit which is not a root of unity.  Let $s$ be the least positive integer such that $\alpha^s=1\bmod \ell$; let $\epsilon=1$ if $\ell=2$ and $0$ otherwise.  Then $$\sum_{i=1}^n v_\ell(1-\alpha^i)\leq C(\alpha)\cdot n,$$ where $$C(\alpha)=\frac{1}{s}\left(v_\ell(\alpha^s-1)+\frac{1}{\ell-1}\right)+\epsilon.$$
\end{lemma}
\begin{proof}
This is Lemma 3.10 of \cite{litt}.
\end{proof}
\subsection{The proof}  We now give the proof of Theorem \ref{ball-bound}.
\begin{proof}[Proof of theorem \ref{ball-bound}]
Let $E_{\bar\rho}$ be the rigid generic fiber of $R_{\bar\rho}$ (this agrees with the notation of Section \ref{deformation-ring-preliminaries} by \cite[Theorems B and F]{chenevier}); as $C$ is an affine curve, $E_{\bar\rho}$ is (non-canonically) analytically isometric to the open unit ball. Then $\rho$ gives a $\overline{\mathbb{Q}_\ell}$-point $z_\rho$ of $E_{\bar\rho}$; we may view $z_\rho$ as a map $$z_\rho: R_{\bar\rho}[1/\ell]\to \overline{\mathbb{Q}_\ell}.$$ 
Let $\mathfrak{m}\subset R_{\bar\rho}[1/\ell]$ be the kernel of this map, and let $S_{\rho}'$ be the completion of $R_{\bar\rho}[1/\ell]$ at $\mathfrak{m}$, so that $S_{\rho}\simeq S_{\rho}'\widehat{\otimes}\overline{\mathbb{Q}_\ell}$.  Let $\mathfrak{m}_{\rho}'$ be the maximal ideal of $S_{\rho}'$. Let $L$ be the residue field of $S_{\rho}'$, so that $S_{\rho}'$ is non-canonically isomorphic to $L[[x_1, \cdots, x_m]]$; let $S_{\rho}^{\text{int}'}=R_{\bar\rho}\otimes_{W(\mathbb{F}_{\ell^r})} \mathscr{O}_L\subset S_{\rho}'$ be the extension of scalars of $R_{\bar\rho}$ to $\mathscr{O}_L$, so that $S_{\rho}^{\text{int}'}$ is non-canonically isomorphic to $\mathscr{O}_L[[x_1, \cdots, x_m]]$, where $\{x_1, \cdots, x_m\}$ generate $\mathfrak{m}_{\rho}'\cap S_{\rho}^{\text{int}'}$.  Choose such an isomorphism. The rigid generic fiber of $S_{\rho}^{\text{int}'}$ is the open unit ball over $L$.

 Choose $\alpha\in \mathbb{Z}_\ell^{\times}$ not a root of unity such that 
\begin{enumerate}
\item there exists $\sigma_\alpha\in G_k$ as in Theorem \ref{pi1-quasiscalar}, and
\item $C(\alpha)$ is minimal among all $\alpha$ satisfying (1).
\end{enumerate}
(Here $C(\alpha)$ is defined as in Lemma \ref{denominator-estimate}.) 

We claim that we may take $N=N(c(\rho), \ell)$ to be any rational number greater than $3C(\alpha)$; note that this may be bounded from above purely in terms of $\ell$ and $c(\rho)$, as the notation suggests.

Let $\sigma_\alpha\in G_k$ be as in Theorem \ref{pi1-quasiscalar}.  For $1\geq r>0$, let $U_r\subset E_{\bar\rho}$ be the closed ball of radius $r$ around $\rho$ --- set-theoretically, this is the set $$\{\tilde\rho\in E_{\bar\rho}\big| |\on{Tr}(\tilde\rho)-\on{Tr}(\rho)|_\ell \leq r\}.$$  Note that $U_{\ell^{-s}}$ is stable under the $\sigma_\alpha$-action on $E_{\bar\rho}$. There is a unique $\pi_1^{\text{\'et}}(C_{\bar k}, \bar c)$-lattice $V_\rho\subset \rho$, up to homothety.  By \cite[Th\'{e}or\`{e}me 1]{carayol}, for $s\in\mathbb{Q}$, $U_{\ell^{-s}}$ is the same as the set of $\tilde\rho$ admitting an $\pi_1^{\text{\'et}}(C_{\bar k}, \bar c)$-stable $\overline{\mathbb{Z}_\ell}$-lattice $V_{\tilde\rho}$ such that $V_{\tilde\rho}/\ell^sV_{\tilde\rho}\simeq V_\rho/\ell^sV_\rho$.  

Let $\mathscr{O}_{U_{\ell^{-s}}}$ be the set of functions on $U_{{\ell^{-s}}}$.  Explicitly, under our chosen isomorphism $S_{\rho}^{\text{int}'}\simeq \mathscr{O}_L[[x_1, \cdots, x_m]]$ and the induced isomorphism $S_\rho'\simeq L[[x_1, \cdots, x_m]]$, $\mathscr{O}_{U_{\ell^{-s}}}\subset S_{\rho}'$ is the subring $$\left\{\sum a_Ix^I \in S_{\rho}'\mid v_\ell(a_I)+|I|\cdot s \to\infty\right\},$$ topologized via the Gauss norm (which makes it into a Tate algebra). (Here $I=(i_1, \cdots, i_m)$ is a multi-index with $i_j\geq 0$, $x^I=x_1^{i_1}\cdots x_m^{i_m}$, and $|I|=\sum_j i_j$.)

Choose a basis of integral $\sigma_\alpha$-eigenvectors of $\mathfrak{m}_\rho'/{\mathfrak{m}_{\rho}'}^2$, and lift it to a set of $\sigma_\alpha$-eigenvectors $\{e_1, \cdots, e_m\}$ of $S_{\rho}'$.  After a linear change of coordinates, we may assume $e_i\equiv x_i\bmod \mathfrak{m}_\rho'$.  We claim that for $s>2C(\alpha)$, we have $e_1, \cdots, e_m\subset \mathscr{O}_{U_{\ell^{-s}}}$.  Indeed, by Lemma \ref{eigenvector-lifting}, we have $$\left(\prod_{i=1}^{2r-1} (1-\alpha^i)\right)\cdot e_j\in S_{\rho}^{\text{int}'}+W^{-2r}\subset S_{\rho}^{\text{int}'}+\mathfrak{m}_{\rho}'^r,$$ and hence if $$e_i=\sum a_{I,i}x^I,$$ we have \begin{equation}\label{bounded-eigenvectors-eqn} v_\ell(a_{I,i})\geq -\sum_{i=1}^{2|I|-1} v_\ell(1-\alpha^i)\geq -(2|I|-1)C(\alpha) \end{equation} by Lemma \ref{denominator-estimate}, whence the claim follows.  By our choice of $e_i$, the coefficient of $x^i$ in the power series expression for $e_i$ is $1$.

Now choose any $s>3C(\alpha)$, and suppose that $\tilde\rho\in U_{\ell^{-s}}$ is arithmetic.  We wish to show that $\tilde\rho\simeq\rho$. We may view $\tilde\rho$ as a map $$z_{\tilde\rho}: \mathscr{O}_{U_{\ell^{-s}}}\to \overline{\mathbb{Q}_\ell},$$ where $v_\ell(z_{\tilde\rho}(x_i))\geq s.$  By the arithmeticity of $\tilde\rho$, there exists an integer $M$ such that $z_{\tilde\rho}$ is equivariant for the natural $\sigma_\alpha^M$-action on $\mathscr{O}_{U_{\ell^{-s}}}$, and the trivial action on $\overline{\mathbb{Q}_\ell}$.  In particular, for each $i$, we have $z_{\tilde\rho}(e_i)=0$.
But we have $$v_\ell(a_{I,i}z_{\tilde\rho}(x^I))=v_\ell(a_{I,i})+v_\ell(z_{\tilde\rho}(x^I))\geq -(2|I|-1)C(\alpha)+s\cdot |I|.$$  For $|I|\geq 2$, then, we have 
\begin{align*}
v_\ell(a_{I,i}z_{\tilde\rho}(x^I))&\geq-(2|I|-1)C(\alpha)+s\cdot |I|\\
&=-(2|I|-1)C(\alpha)+s+s(|I|-1)\\
&> -(2|I|-1)C(\alpha)+s+3C(\alpha)(|I|-1)\\
&=s+C(\alpha)(|I|-2)\\
&\geq s.
\end{align*}
Hence by the ultrametric inequality, $z_{\tilde\rho}(e_i)=0$ implies $z_{\tilde\rho}(x_i)=0$.  Thus $\ker(z_{\tilde\rho})=\ker(z_\rho)$, and hence $\rho=\tilde\rho$, as desired.
\end{proof}
\begin{remark}
From the proof, we see that we may take $N$ to be any rational number greater than $3C(\alpha)$, where $\alpha\in\mathbb{Z}_\ell^\times$ is such that there exists $\sigma_\alpha$ as in Theorem \ref{pi1-quasiscalar} (or Lemma \ref{h1-quasiscalar}).
\end{remark}
\subsection{The case of representations with finite image}
Finally, we prove Theorem \ref{finite-image-thm}. The proof is a slight variant of the main theorem of \cite{litt}, with additional input from a result of Serre, of which we were unaware at the time of writing \cite{litt}.
\begin{proof}[Proof of Theorem \ref{finite-image-thm}]
By passing to the cover of $X$ defined by $\ker(\pi_1^{\text{\'et}}(X_{\bar k}, \bar x)\to G$, we may assume $\rho$ is trivial.  Now the proof is essentially identical to that of Theorem 1.2 of \cite{litt}.  While Remark 4.3 of \cite{litt} indicates that the constant $N$ of that theorem depends on the index of the image of $G_k\to GL(H^1(X_{\bar k}, \mathbb{Z}_\ell)$ in the $\mathbb{Z}_\ell$-points of its Zariski closure; in fact it depends only on the set of $Z$ of $\alpha\in \mathbb{Z}_\ell^{\times}$ such that there exists $\sigma_\alpha\in G_k$ such that $\sigma_\alpha$ acts on $\on{gr}^{i}_WH^1(X_{\bar k}, \mathbb{Z}_\ell)$ via $\alpha^i\cdot \on{Id}$.  

By Remark 3.11 of \cite{litt}, it suffices to show that the index of $Z$ in $\mathbb{Z}_{\ell}^\times$ is bounded independent of $\ell$.  The case where $X$ has genus zero is covered in \cite{litt}, so we assume the genus of $X$ is at least one.

Let $\overline{X}$ be the unique smooth proper geometrically connected $k$-curve containing $X$, and let $D=\overline{X}\setminus X$. Replace $k$ with a finite extension, so that $D=\{x_1, \cdots, x_m\}$ for $x_i\in \overline{X}(k)$.  

Now we claim that $Z$ contains the set of $\alpha\in \mathbb{Z}_\ell^{\times}$ so that there exists $\sigma_\alpha'\in G_k$ such that $\sigma_\alpha'$ acts on $\on{gr}^{1}_W H^1(X_{\bar k}, \mathbb{Z}_\ell)$ via $\alpha\cdot\on{Id}$.  Indeed, $\on{gr}^2_WH^1(X_{\bar k}, \mathbb{Z}_\ell)$ is a direct sum of copies of the cyclotomic character, so by Poincar\'e duality, any such $\sigma_\alpha'$ in fact acts on $\on{gr}^2_WH^1(X_{\bar k}, \mathbb{Z}_\ell)$ as desired.

But now the index of $Z$ in $\mathbb{Z}_\ell^\times$ is bounded independent of $\ell$ by a result of Serre \cite[Lettre \`{a} Ken Ribet, p.~60]{serre-oeuvres}, using that $\on{gr}^{1}_W H^1(X_{\bar k}, \mathbb{Z}_\ell)$ is dual to the $\ell$-adic Tate module of the Jacobian of $\overline{X}$.
\end{proof}

\begin{remark}\label{tate-remark}
It is natural to conjecture that if $\rho_\ell$ is a compatible system of irreducible lisse sheaves on a curve, $c(\rho_\ell)$ is bounded independent of $\ell$. More precisely, fix $p,q\in \mathbb{Z}_{\geq 0}$. If $k$ is a number field and $f: X\to Y$ is a smooth proper morphism over $k$, let $H_\ell\subset \mathbb{Z}_\ell^\times$ be the set of $\alpha\in \mathbb{Z}_{\ell}^\times$ such that there exists $\sigma_\alpha\in G_k$ with $\sigma_\alpha$ acting on $\on{gr}^i_W H^p(Y_{\bar k}, R^qf_*\mathbb{Q}_\ell)$ via $\alpha^i\cdot \on{id}$.  We conjecture that the index of $H_\ell\subset \mathbb{Z}_\ell^\times$ is bounded independent of $\ell$, and indeed that this index is equal to $1$ for almost all $\ell$.

Assuming the Tate conjecture, one may show this index is uniformly bounded via the argument of \cite[\S 2.3]{wintenberger}.  This gives, by the proof of Theorem \ref{ball-bound}, a much stronger version of Theorem \ref{ball-bound}. It implies (on the Tate conjecture) that if $\rho_\ell$ is a compatible system of $\ell$-adic representations arising from geometry (i.e.~the monodromy representation underlying $R^qf_*\mathbb{Q}_\ell$ as above), then the constants $N(c(\rho_\ell), \ell)$ of Theorem \ref{ball-bound} tend to zero as $\ell\to\infty$. 
\end{remark}
\bibliographystyle{alpha}
\bibliography{monodromy-2-bibtex}

\end{document}